\documentclass[english, 12pt,leqno]{amsart}

\usepackage{a4,
amsfonts,
amsmath,
amssymb,
amstext,
enumitem,
color,
url}

\usepackage{tikz}
\tikzstyle{vertex}=[circle, draw, inner sep=0pt, minimum size=8pt]
\newcommand{\vertex}{\node[vertex]}

\newtheorem{thm}{Theorem}[section]
\newtheorem{cor}[thm]{Corollary}
\newtheorem{lem}[thm]{Lemma}
\newtheorem{prop}[thm]{Proposition}

\theoremstyle{definition}

\newtheorem{exe}[thm]{Example}
\newtheorem{rem}[thm]{Remark}
\newtheorem{defn}[thm]{Definition}
\newtheorem{comp}[thm]{Compilation of spectral measures}

\newcommand{\N}{\mathbf N}
\newcommand{\Z}{\mathbf Z}
\newcommand{\R}{\mathbf R}
\newcommand{\Q}{\mathbf Q}
\newcommand{\C}{\mathbf C}
\newcommand{\K}{\mathbf K}

\newcommand{\Hi}{\mathcal H}
\newcommand{\Ki}{\mathcal K}
\newcommand{\sca}{\hskip-.05cm\mid\hskip-.05cm}

\title[Spectral measures and dominant vertices in graphs]
{Spectral measures and dominant vertices
\\
in graphs of bounded degree}



\thanks{
The authors acknowledge support of the Swiss NSF grants 200020-178828 and 200020-20040.
}

\date{10 August 2023}

\author{Claire Bruchez, Pierre de la Harpe, and Tatiana Nagnibeda}

\address{Claire Bruchez, Pierre de la Harpe, and Tatiana Nagnibeda
\newline
Section de math\'ematiques,
Universit\'e de Gen\`eve,
Uni Dufour,
\newline
24 rue du G\'en\'eral Dufour,
Case postale 64,
CH--1211 Gen\`eve 4.
}

\email{Claire.Bruchez@etu.unige.ch
\hskip.5cm Pierre.delaHarpe@unige.ch
\newline Tatiana.Smirnova-Nagnibeda@unige.ch}

\begin{document}

\begin{abstract}
A graph $G = (V, E)$ of bounded degree
has an adjacency operator~$A$
which acts on the Hilbert space $\ell^2(V)$.
There are different kinds of measures of interest
on the spectrum $\Sigma (A)$ of $A$.
In particular, each vector $\xi \in \ell^2(V)$ defines
a local spectral measure $\mu_\xi$ at $\xi$ on $\Sigma (A)$;
therefore each vertex $v \in V$ defines
a vector $\delta_v \in \ell^2(V)$
and the associated measure $\mu_v$ on~$\Sigma (A)$.
A vertex $v$ is dominant if, for all $w \in V$, the measure $\mu_w$
is absolutely continuous with respect to $\mu_v$
(it then follows that, for all $\xi \in \ell^2(V)$, the measure $\mu_\xi$
is absolutely continuous with respect to $\mu_v$).
The main object of this paper is to show that all possibilities occur:
in some graphs, for example in vertex-transitive graphs,
all vertices are dominant;
in other graphs, only some vertices are dominant;
and there are graphs without dominant vertices at all.
\end{abstract}

\maketitle

\section{Introduction}
\label{SectionIntro}

A \textbf{graph} $G$ with set of vertices $V$ is defined
combinatorially by its set of edges $E$,
which is a subset of the set of unordered pairs of vertices,
and analytically by its \textbf{adjacency operator} $A$,
which acts on the Hilbert space $\ell^2(V)$.
We consider only countable graphs with bounded degree, i.e.,
graphs such that $V$ is countable (finite or infinite)
and such that the number
of edges incident to a vertex $v$
is bounded by some constant independent of $v$.
In this case the Hilbert space $\ell^2(V)$ is separable
and the operator $A$ defined by
$$
(A \xi)(v) = \sum_{\substack{w \in V \text{ such that } \\ \{v,w\} \in E}} \xi (w)
\hskip.5cm \text{for all} \hskip.2cm
\xi \in \ell^2(V)
\hskip.2cm \text{and} \hskip.2cm
v \in V 
$$
is bounded and self-adjoint. 
Considerable attention has been given to the relations
between the properties of $G$ relevant to graph theory and random walks,
and the spectral properties of $A$ such as its spectrum
or various natural measures on it;
see for example \cite{Kest--59}, \cite{Moha--82}, \cite{GoMo--88}, \cite{MoWo--89}
for infinite graphs
and \cite{CoSi--57}, \cite{CvDS--80}, \cite{Bigg--93}, \cite{GoRo--01}, \cite{BrHa--12}
for finite graphs.
\par

Our motivation for this work is to emphasize
that there are \textbf{several kinds of spectral measures}
on the spectrum of an adjacency operator,
and that they should not be confused with each other
(see Compilation~\ref{compilation}).
In the seminal paper by Kesten \cite{Kest--59},
the central role is played by the measure $\mu_v$
on the spectrum of $A$ defined as follows:
choose a vertex $v \in V$,
and denote by $\delta_v \in \ell^2(V)$ its characteristic function;
denote by $A$ the adjacency operator of the graph,
by $\Sigma (A)$ its spectrum,
by $\mathcal C (\Sigma(A))$ the space of all continuous functions $\Sigma (A) \to \C$;
define the \textbf{vertex spectral measure of $A$ at $v$} by
$$
\mu_v \, \colon \mathcal C (\Sigma(A)) \to \C ,
\hskip.5cm
f \mapsto \langle f(A) \delta_v \sca \delta_v \rangle ,
$$
where $f(A)$ is defined by operator calculus.
Two important and well-known properties of $\mu_v$
are that, for any non-negative integer $n$,
the moment of order $n$ of $\mu_v$ counts closed walks in the graph $G = (V, E)$,
$$
\int_{\Sigma (A)} t^n d\mu_v (t) = \langle A^n \delta_v \sca \delta_v \rangle =
\text{number of closed walks in $G$ of length $n$ at $v$},
$$
and that $\max \{ t \in \text{closed support of} \hskip.2cm \mu_v \} = \Vert A \Vert$
(see Proposition~\ref{countloops}).
The graphs of main interest in Kesten's paper
are vertex-transitive (i.e., for any $v, w \in V$,
there exists an automorphism $\alpha$ of $G$ such that $\alpha (v) = w$);
in this case, $\mu_v$ does not depend on the choice of $v$,
and this makes it clear that $\mu_v$ is the appropriate measure to consider.
However, in general, $\mu_v$ does depend on $v$;
indeed, for another vertex $w$,
the measure $\mu_w$ need not be absolutely continuous with respect to~$\delta_v$.
To our knowledge, this fact has not been properly addressed in the literature,
and this was our initial motivation.
\par

We focus attention on certain particular vertices,
that we call \textbf{dominant vertices}:
a vertex $v \in V$ is dominant if, for any other vertex $w \in V$,
the measure $\mu_w$ is absolutely continuous with respect to $\mu_v$
(Definition~\ref{defdomv})\footnote{The
notion of a \emph{dominant vertex} has nothing to do with
that of a \emph{dominating set} of vertices, which is a set of vertices such that
any vertex in the graph is at distance at most $1$ of the set.}.
\par

For example, all vertices are dominant in vertex-transitive-graphs
(Proposition~\ref{transimpliesdom})
and more generally in walk-regular graphs
(Proposition~\ref{wrg}).
When $G$ is a finite graph, a vertex $v$ is dominant
if the measure defined by $v$ on the spectrum of $A$
charges all eigenvalues of $A$,
equivalently if the orthogonal projection of $\delta_v$
on every eigenspace of $A$ is not zero
(see Proposition~\ref{domvertfg},
for this and other characterizations of dominant vertices in finite graphs).
The main point of what follows is to provide examples of graphs, both infinite and finite,
which show the various possibilities: all vertices may be dominant,
or only some of them, 
or just one of them,
or none at all.
\par

It has been conjectured that ``most finite graphs''
have characteristic polynomials irreducible over $\Q$
(\cite{ORWo--19} and \cite[Conjecture 7.3]{LiSi--22});
if true, this would imply that, in ``most finite graphs'',
all vertices are dominant (see Proposition~\ref{PropIrrQ}).
There are several examples of connected finite graphs without dominant vertices
in Section~\ref{SectionFExamples},
and of a connected infinite graph with only two dominant vertices
in Example~\ref{Only2dom}.
We haven't been able to find a bounded degree connected infinite graph
without dominant vertices; for this open question, the Pascal graph of \cite{Quin--09}
could be an example to study.

\vskip.2cm

Section~\ref{SectionDomVectors} is a reminder of spectral analysis
for bounded self-adjoint operators on Hilbert spaces,
with an emphasis on dominant vectors and cyclic vectors.
\par

In Section~\ref{SectionDomVertices}, we review
adjacency operators of graphs of bounded degree,
and we define vertex spectral measures and dominant vertices.
Proposition~\ref{transimpliesdom} shows that
all vertices are dominant in vertex-transitive graphs;
this applies to Cayley graphs, in particular to
the infinite line, to square grids, and to regular trees.
More generally, all vertices are dominant in walk-regular graphs
(Proposition~\ref{wrg}).
Later, it is shown that all vertices in the infinite ray (which has trivial automorphism group)
are dominant, indeed cyclic (Proposition~\ref{PropExRay});
the proof we know uses classical properties of Chebyshev polynomials.
By contrast, there is a minor modification of the infinite ray in which
almost all vertices are not dominant (Example~\ref{Only2dom}).
\par

The next two sections are for finite graphs.
In Section~\ref{SectionFG}, on linear algebra,
it is shown how vertex spectral measures and dominant vertices
can be characterized for finite graphs.
Then Section~\ref{SectionFExamples} shows examples
of connected finite graphs and their dominant vertices:
finite paths, complete bipartite graphs, finite graphs without dominant vertices at all
(including trees and regular connected graphs),
distance-regular graphs, and Coxeter graphs.
\par

Section~\ref{SectionFE} is a continuation of Section~\ref{SectionDomVertices},
with other examples of infinite graphs (infinite ray and stars)
in which all vertices (or possibly all but one) are dominant.
Section~\ref{S+EofJc} is a technical interlude
on a particular Jacobi matrix
(its spectrum, its eigenvalues, and its cyclic vectors)
which appears in the example of infinite stars (Theorem~\ref{exStars}).
\par

Similar questions on dominant vertices can, and should be asked
about other adjacency-type operators on graphs
(Laplacian, normalized adjacency, Mar\-kov, etc.).
In Section~\ref{sectionMarkov}, me address briefly
the \textbf{Markov operator} $M$ of a graph $G = (V, E)$,
defined as the matrix of transition probabilities of the simple random walk on the graph
(see for example \cite[Section 2]{MoWo--89}, where $M$ is called the transition operator).
A vertex $v \in V$ is \emph{$M$-dominant} if the vector $\delta_v$ is dominant for $M$.
When $G$ is regular of degree $k$, we have $M = \frac{1}{k}A$,
and a vertex is therefore $M$-dominant
if and only if it is $A$-dominant (= dominant).
However, for non-regular graphs,
the $A$-dominant vertices need not be the same as the $M$-dominant vertices;
it is easy to verify that it is so in finite paths, for example.
\par

For the reader interested essentially in finite graphs,
the important parts are Sections~\ref{SectionFG} and~\ref{SectionFExamples}.

\newpage

\section{Spectral measures, local spectral measures,
\\
and dominant vectors for a bounded self-adjoint operator}
\label{SectionDomVectors}

Let $\Hi$ be a Hilbert space (here always over the complex numbers, and not $\{0\}$)
and $X$ a bounded self-adjoint operator on $\Hi$.
The spectrum $\Sigma (X)$ is the set of complex numbers $\lambda$
such that $\lambda - X$ is not invertible; it is a nonempty compact subset of~$\mathbf R$,
contained in the interval $[ -\Vert X \Vert, \Vert X \Vert]$.
The \textbf{spectral measure} of $X$ is the projection-valued measure
$E_X \, \colon \mathfrak B_{\Sigma (X)} \to {\rm Proj} (\Hi)$ such that
\begin{equation}
\label{EqProjSpecMeas}
X = \int_{\Sigma (X)} t dE_X(t) ,
\end{equation}
where $\mathfrak B_{\Sigma (X)}$ is the $\sigma$-algebra of Borel sets on $\Sigma (X)$
and ${\rm Proj} (\Hi)$ is the set of orthogonal projections of $\Hi$ onto closed subspaces.
A \textbf{scalar-valued spectral measure} for $X$
is a finite positive measure $\mu \, \colon \mathfrak B_{\Sigma (X)} \to \mathbf R_+$
such that $\mu(B) = 0$ (the zero number)
if and only if $E_X(B) = 0$ (the zero projection on $\Hi$)
for any Borel subset $B$ of $\Sigma (X)$;
this terminology is that of \cite{Conw--07}.
The existence of a scalar-valued spectral measure for $X$
is guaranteed by Proposition~\ref{dominant};
it follows from the definition that two scalar-valued spectral measures for $X$
are equivalent to each other.
\par

For any bounded Borel complex-valued function on the spectrum,
$f \in \mathcal B (\Sigma (X))$,
functional calculus provides a bounded operator $f(X)$ on $\Hi$;
this holds in particular for any continuous function on the spectrum,
$f \in \mathcal C (\Sigma (X))$.
The local spectral measure of $X$ at two vectors $\xi$ and $\eta$ in $\Hi$
is the complex measure $\mu_{\xi, \eta}$ on $\Sigma (X)$ defined by
$$
\int_{\Sigma(X)} f(t) d\mu_{\xi, \eta} (t) = \langle f(X) \xi \sca \eta \rangle
\hskip.5cm \text{for all} \hskip.2cm
f \in \mathcal C (\Sigma (X)) 
$$
or equivalently by
$$
\mu_{\xi, \eta}(B) = \langle E_X(B) \xi \sca \eta \rangle
\hskip.5cm \text{for all Borel subset} \hskip.2cm
B \subset \Sigma (X) .
$$
In particular
the \textbf{local spectral measure} of $X$ at a vector $\xi \in \Hi$
is the positive measure on $\Sigma (X)$ denoted by $\mu_\xi$
(rather than $\mu_{\xi, \xi}$) defined by
\begin{equation}
\label{EqLocalSpecMeas}
\int_{\Sigma(X)} f(t) d\mu_\xi(t) = \langle f(X) \xi \sca \xi \rangle
\hskip.5cm \text{for all} \hskip.2cm
f \in \mathcal C (\Sigma (X)) .
\end{equation}
The norm of $\mu_\xi$ is $\mu_\xi(\Sigma(X)) = \Vert \xi \Vert^2$.
It follows from the definitions that $\mu_\xi$ is absolutely continuous
with respect to any scalar-valued spectral measure for~$X$.
Note that $\mu_\xi$ can equally be viewed as a positive measure on $\mathbf R$
with closed support contained in $\Sigma (X)$.
\par

A vector $\xi \in \Hi$ is \textbf{dominant} for $X$
if $\mu_\eta \prec \hskip-.1cm \prec \mu_\xi$
for all $\eta \in \Hi$,
in other words if the measure $\mu_\eta$ is absolutely continuous
with respect to $\mu_\xi$ for all $\eta \in \Hi$,
i.e., if the measure $\mu_\eta$ is dominated by the measure $\mu_\xi$
for all $\eta \in \Hi$.
\par

For examples of dominant vectors, see Proposition~\ref{c+m} below.
For an easy non-example, consider a self-adjoint operator $X$
which has an eigenvalue $\lambda$
and which is not $\lambda$ times the identity;
an eigenvector of eigenvalue $\lambda$ is not a dominant vector for $X$.
\par

Terminology varies from one author to the other.
A dominant vector for $X$ is 
\emph{maximal relative to $X$} in \cite[Chap.\ X, Section~5, Definition~6]{DuSc--63},
it is a \emph{vector of maximal type} in \cite{BoSm--20},
and a \emph{separating vector}
for the von Neumann algebra generated by $X$ in \cite{Dixm--69}.
For $\xi$ dominant, the measure $\mu_\xi$
is a scalar-valued spectral measure for $X$;
conversely, any scalar-valued spectral measure
is a measure $\mu_\xi$ for some dominant $\xi \in \Hi$
\cite[Theorem IX.8.9 and Lemma IX.8.6]{Conw--07}.

\begin{prop}[\textbf{existence and characterizations of dominant vectors
for self-adjoint operators}]
\label{dominant}
Let $X$ be a bounded self-adjoint operator on a separable Hilbert space $\Hi$.
Let $\mathcal B (\Sigma (X))$ and $E_X$ be as above.
\begin{enumerate}[label=(\arabic*)]
\item\label{1DEdominant}
There exist dominant vectors for $X$.
More precisely, for any $\eta \in \Hi$, there exists a dominant vector $\xi$ for $X$
such that $\eta$ is in the closed linear span of $\{ X^n \xi \}_{n \in \N}$.
\item\label{2DEdominant}
A vector $\xi \in \Hi$ is dominant for $X$ if and only if, for $f \in \mathcal B (\Sigma (X))$,
the equality $f(X) \xi = 0$ implies $f(X) = 0$.
\item\label{3DEdominant}
A vector $\xi \in \Hi$ is dominant for $X$ if and only if,
for any Borel subset $B$ of $\Sigma(X)$, we have
$\mu_\xi(B) = 0$ if and only if $E_X(B) = 0$,
i.e., if and only if $\mu_\xi$ is a scalar-valued spectral measure for $X$.
\end{enumerate}
\end{prop}

\begin{proof}[References for the proof]
For~\ref{1DEdominant}
and~\ref{2DEdominant},
see~\cite{Sim4--15}, Lemma 5.4.7 and Problem~3 of \S~5.4.
For~\ref{3DEdominant},
see~\cite{Conw--07}, Theorem IX.8.9.
\end{proof}

\begin{prop}
\label{dom=dom}
Let $\Hi$, $X$, and $E_X$ be as in Proposition~\ref{dominant}.
Let $\xi \in \Hi$ and let $(\varepsilon_i)_{i \ge 1}$ be an orthonormal basis of $\Hi$.
Denote by $\mu_\xi$ and $\mu_i$ the local spectral measures of $X$
at $\xi$ and $\varepsilon_i$.
\par
Then the vector $\xi$ is dominant for $X$ if and only if
the measure $\mu_\xi$ dominates~$\mu_i$ for all $i \ge 1$.
\end{prop}

\begin{proof}
If $\xi$ is dominant,
then $\mu_i \prec \hskip-.1cm \prec \mu_\xi$ for all $i \ge 1$,
by definition.
Conversely, suppose that $\mu_i \prec \hskip-.1cm \prec \mu_\xi$ for all $i \ge 1$.
Let $B$ be a Borel subset of $\Sigma (X)$.
If $\mu_\xi (B) = 0$,
then $\mu_i(B) = \langle E_X(B) \varepsilon_i \sca \varepsilon_i \rangle
= \Vert E_X(B) \varepsilon_i \Vert^2 = 0$
for all $i \ge 1$, namely $E_X(B) \varepsilon_i = 0$ for all $i \ge 1$,
hence $E_X(B) = 0$.
It follows that $\xi$ is dominant for~$X$
(see~\ref{3DEdominant} in Proposition~\ref{dominant}).
\end{proof}

Let $X$ be a bounded self-adjoint operator on a Hilbert space $\Hi$.
A vector $\xi \in \Hi$ is \textbf{cyclic} for $X$
if the closed linear span of $\{X^n \xi : n \in \N\}$
is the whole of~$\Hi$.
(Some authors choose another definition for more general operators,
but all definitions agree in case $X$ is self-adjoint;
see~\cite[Page 289, \S~5.1]{Sim4--15}.)
Note that, if $X$ has any cyclic vector at all,
then the space $\Hi$ is separable.
\par

The following proposition is well-known and straightforward:

\begin{prop}
\label{dcyclic}
For a self-adjoint operator $X$, any cyclic vector is dominant.
\end{prop}

\begin{proof}
Let $\xi$ be a cyclic vector for $X$.
Let $f$ in $\mathcal B (\Sigma (X))$ be a bounded Borel function on the spectrum of $X$.
If $f(X) \xi = 0$, then $f(X) X^n \xi = X^n f(X) \xi = 0$ for all $n \ge 0$,
hence $f(X) \eta = 0$ for all $\eta$ in the closed linear span of $\{X^n \xi : n \in \N\}$,
hence $f(X) = 0$.
(We have used the characterization of dominant vectors given in~\ref{2DEdominant}
of Proposition~\ref{dominant}).
\end{proof}

There are dominant vectors which are not cyclic,
for example when $X$ is a self-adjoint operator on a finite-dimensional space
which has multiple eigenvalues
(more generally when $X$ is a self-adjoint operator which is not multiplicity-free).

\vskip.2cm

We will use the following facts on dominant vectors for simple multiplication operators.
Let $\Sigma$ be a nonempty compact subset of the real line,
$\mathfrak B_\Sigma$ the Borel $\sigma$-algebra of $\Sigma$,
and $\mu$ a finite Borel measure on $(\Sigma, \mathfrak B_\Sigma)$,
of closed support $\Sigma$.
The \textbf{multiplication operator} $M_\mu$
is here the operator defined on $L^2(\Sigma, \mu)$ by
$$
(M_\mu \xi) (t) = t \xi (t) 
\hskip.5cm \text{for all} \hskip.2cm
\xi \in L^2(\Sigma, \mu)
\hskip.2cm \text{and} \hskip.2cm
t \in \Sigma .
$$
It is a bounded self-adjoint operator
with norm $\max_{t \in \Sigma} \vert t \vert$
and with spectrum~$\Sigma$.
The function $1_\Sigma \in L^2(\Sigma, \mu)$ of constant value $1$
is a cyclic vector for~$M_\mu$.
More is true:

\begin{prop}[\textbf{cyclic and dominant vectors for multiplication operators}]
\label{c+m}
Let $M_\mu$ be a multiplication operator as above and $\xi \in L^2(\Sigma, \mu)$.
The following three properties are equivalent:
\begin{enumerate}[label=(\roman*)]
\item\label{iDEc+m}
$\xi$ is cyclic for $M_\mu$,
\item\label{iiDEc+m}
$\xi$ is dominant for~$M_\mu$,
\item\label{iiiDEc+m}
$\xi$ is such that $\mu (\{ t \in \Sigma : \xi(t) = 0 \}) = 0$.
\end{enumerate}
\end{prop}

\begin{proof}
Property~\ref{iDEc+m} implies Property~\ref{iiDEc+m}
by Proposition~\ref{dcyclic}.
\par

Suppose that~\ref{iiiDEc+m} does not hold.
Define $f \in \mathcal B (\Sigma)$ by
$f(t) = 1$ when $\xi(t) = 0$ and $f(t) = 0$ when $\xi(t) \ne 0$;
then $f(M_\mu) \ne 0$ and $f(M_\mu)\xi = 0$,
hence~\ref{iiDEc+m} does not hold.
This shows that~\ref{iiDEc+m} implies~\ref{iiiDEc+m}.
\par

Suppose that~\ref{iiiDEc+m} holds.
Denote by $L^2_\xi$ the closed subspace of $L^2(\Sigma, \mu)$
spanned by $\{ (M_\mu)^n\xi \}_{n \in \N}$.
We have to show that $L^2_\xi = L^2(\Sigma, \mu)$.
For this we consider $\eta \in L^2(\Sigma, \mu)$
orthogonal to $(M_\mu)^n\xi$ for all $n \ge 0$,
and we will show that $\eta = 0$.
The condition $\eta \perp (M_\mu)^n\xi$ reads
$$
\int_\Sigma t^n \xi(t) \overline{\eta(t)} d\mu(t) = 0
\hskip.5cm \text{for all} \hskip.2cm
n \ge 0 ,
$$
so that
$$
\int_\Sigma p(t) \xi(t) \overline{\eta(t)} d\mu(t) = 0
\hskip.5cm \text{for all} \hskip.2cm
p \in \mathcal P (\Sigma) ,
\leqno{(\sharp)}
$$
where $\mathcal P (\Sigma)$ is the algebra of functions on $\Sigma$
which are restrictions to $\Sigma$ of polynomial functions on $\R$.
Note that $\xi \overline{\eta}$ is in $L^1(\Sigma, \mu)$
and that $L^\infty(\Sigma, \mu)$ is the dual of the Banach space $L^1(\Sigma, \mu)$.
Note also that $\mathcal P (\Sigma)$ is w*-dense in $L^\infty(\Sigma, \mu)$,
because the algebra $\mathcal P (\Sigma)$ is dense in the Banach algebra $\mathcal C (\Sigma)$,
by the Stone--Weierstrass theorem,
and because the natural image of $\mathcal C (\Sigma)$ in $L^\infty(\Sigma, \mu)$
is w*-dense, see~\cite[Corollary 4.5.3]{Doug--72}.
Consequently, it follows from ($\sharp$) that
$\xi \overline{\eta} = 0$ in $L^1(\Sigma, \mu)$,
hence that $\xi (t) \overline{\eta (t)} = 0$ for $\mu$-almost all $t \in \Sigma$.
The last equality and Condition~\ref{iiiDEc+m}
imply that $\eta (t) = 0$ for $\mu$-almost all $t \in \Sigma$,
so that $\eta = 0$.
This shows that~\ref{iiiDEc+m} implies~\ref{iDEc+m}.
\end{proof}

\begin{cor}
\label{c=d}
Let $X$ be a bounded self-adjoint operator
on a separable Hilbert space $\Hi$.
Assume that there exists a cyclic vector for $X$.
\par
Then a vector in $\Hi$ is dominant for $X$ if and only if it is cyclic for $X$.
\end{cor}

\begin{proof}
Let $\Sigma$ denote the spectrum of $X$,
let $\xi_0$ be a cyclic vector for $X$, 
and let $\mu$ denote the local spectral measure of $X$ at $\xi_0$.
Then $X$ is unitarily equivalent to the operator $M_\mu$
of Proposition~\ref{c+m}; more precisely,
there exists a surjective isometry $U \, \colon \Hi \to L^2(\Sigma, \mu)$
such that $X = U^* M_\mu U$
and $U \xi_0$ is the function on $\Sigma$ of constant value $1$.
This can be viewed as a form of the spectral theorem:
see \cite[Theorem~5.1.7]{Sim4--15}.
Note that the equality $UX = M_\mu U$ determines for each $n \ge 0$
the image of $X^n \xi_0$ by $U$; indeed
$(U X^n \xi_0)(t) = (M_\mu^n U \xi_0)(t) = t^n(U\xi_0)(t)$
for all $t \in \Sigma$,
so that $U X^n \xi_0$ is the function $t \mapsto t^n$ in $L^2(\Sigma, \mu)$.
\par

Since a vector $\xi \in \Hi$ is dominant [respectively cyclic] for $X$
if and only if $U\xi \in L^2(\Sigma, \mu)$ is dominant [respectively cyclic] for $M_\mu$,
the corollary follows from Proposition~\ref{c+m}.
\end{proof}

The next proposition is again straightforward,
but will play a crucial role in the proof of Theorem~\ref{exStars}.

\begin{prop}
\label{Sumdominant}
Let $\Hi = \bigoplus_{i \in I} \Hi_i$ be a Hilbert space given as
an orthogonal direct sum of a sequence of non-trivial subspaces,
where $I = \{1, 2, \hdots \}$ is countable (finite or infinite).
For each $i \in I$, let $\xi_i \ne 0$ be a vector in $\Hi_i$;
assume that $\sum_{i \in I} \Vert \xi_i \Vert^2 < \infty$,
and set $\xi = \sum_{i \in I} \xi_i \in \Hi$.
\par
Let $X$ be a self-adjoint operator on $\Hi$.
Assume that $\Hi_i$ is invariant by $X$,
and denote by $X_i$ the restriction of $X$ to~$\Hi_i$, for all $i \in I$.
\begin{enumerate}[label=(\roman*)]
\item\label{iDESumdominant}
If $\xi_i$ is dominant for $X_i$ for all $i \in I$,
then $\xi$ is dominant for $X$.
\item\label{iiDESumdominant}
If $\vert I \vert \ge 2$, then $\xi_i$ is not cyclic for $X$, for all $i \in I$.
\end{enumerate}
\end{prop}

\begin{proof}
\ref{iDESumdominant}
Let $f$ be a bounded Borel function on the spectrum of $X$ such that $f(X) \xi = 0$.
Note that $f(X) \xi = \left( f(X_i) \xi_i \right)_{i \in I}$.
If $\xi_i$ is dominant for $X_i$ for all $i \in I$, then $f(X_i) = 0$ for all $i \in I$,
hence $f(X) = 0$, therefore $\xi$ is dominant for $X$.
\par

\ref{iiDESumdominant}
Let $i \in I$.
The closed linear span of $\{X^n \xi_i : n \in \N\}$
is contained in $\Hi_i$, hence is a proper subspace of $\Hi$.
\end{proof}

\section{Vertex spectral measures and dominant vertices
\\
for adjacency operators of graphs}
\label{SectionDomVertices}

Let $G = (V, E)$ be a graph with $V$ countable, and of bounded degree.
Denote by $A$ the adjacency operator of $G$ and by $\Sigma (A)$ its spectrum.
The Hilbert space $\ell^2(V)$ has a canonical basis $(\delta_v)_{v \in V}$
of characteristic functions of vertices,
$\delta_v (v) = 1$ and $\delta_v(w) = 0$ for $w \in V \smallsetminus \{v\}$.
With respect to this basis, the adjacency operator $A$ of $G$
has matrix coefficients
$$
\langle A \delta_w \sca \delta_v \rangle =
\left\{\begin{aligned}
1 \hskip.2cm &\text{if} \hskip.2cm \{v, w\} \in E 
\\
0 \hskip.2cm &\text{otherwise} 
\end{aligned}\right.
$$
(we think of $v$ as a row index and of $w$ as a column index for the matrix of $A$).
Observe that $A = 0$ if and only if $G$ has no edge.
\par

The following proposition describes the convex hull of the spectrum $\Sigma (A)$.
Though most of it is quite standard (see \cite{Moha--82}
and \cite[Lemmas 10.1 and 1.7]{Woes--00}),
we did not find a convenient reference
for the precise inequalities of Claim~\ref{1DE1stqual},
indeed for $m_A < 0 < M_A = \Vert A \Vert$,
so that we decided to include Proposition~\ref{1stqual} here.

\begin{prop}[\textbf{smallest and largest spectral values}]
\label{1stqual}
Let $G = (V, E)$, $A$, and $\Sigma (A)$ be as above.
We assume that $E \ne \emptyset$.
\begin{enumerate}[label=(\roman*)]
\item\label{1DE1stqual}
We have $\Vert A \Vert \ge 1$ and
$$
-\Vert A \Vert \le \min \{ t \in \Sigma (A) \} \le -1
< 1 \le \max \{ t \in \Sigma (A) \} = \Vert A \Vert .
$$
\item\label{2DE1stqual}
If $G$ is bipartite, then $\Sigma (A)$ is symmetric with respect to zero,
so that in particular $\min \{ t \in \Sigma(A) \} = - \Vert A \Vert$.
\end{enumerate}
\end{prop}

\noindent \emph{Remarks.}
Let $X$ be a bounded self-adjoint operator on a Hilbert space $\Hi$.
Set
$$
\begin{aligned}
M_X = \sup \{ t \in \R &\mid t = \langle X\xi \mid \xi \rangle
\hskip.2cm \text{for some} \hskip.2cm
\xi \in \Hi , \hskip.1cm \Vert \xi \Vert = 1 \} ,
\\
m_X = \inf \{ t \in \R &\mid t = \langle X\xi \mid \xi \rangle
\hskip.2cm \text{for some} \hskip.2cm
\xi \in \Hi , \hskip.1cm \Vert \xi \Vert = 1 \} .
\end{aligned}
$$
It is well-known that
$$
M_X = \max \{ t \in \Sigma (X) \} , \hskip.5cm
m_X = \min \{ t \in \Sigma (X) \} ,
$$
so that the closed interval $\mathopen[ m_X, M_X \mathclose]$
is the convex hull of the spectrum $\Sigma (X)$.
Moreover $-\Vert X \Vert \le m_X \le M_X \le \Vert X \Vert$
and $\Sigma (X)$ contains at least one of $\Vert X \Vert, -\Vert X \Vert$;
see for example \cite[Problem 171]{Halm--67};
but $M_X$ need not be~$\ge 0$ and $m_X$ need not be $\le 0$.
\par

When $G$ is connected, we have moreover $\max \{ t \in \Sigma (A) \} \ge 2$
unless $G$ is in a well-known list of finite graphs, the finite Coxeter graphs;
the list is for example in \cite[Section~3.1.1]{BrHa--12},
see also Example~\ref{exCoxFinite} below.
Similarly, finite connected graphs with $\min \{ t \in \Sigma (A) \} \ge -2$
are very special, and well understood; see~\cite[Section 8.4]{BrHa--12};
for infinite connected graphs of bounded degree, 
$\min \{ t \in \Sigma (A) \} \le -2$.
\par

A connected finite graph such that $-\Vert A \Vert$ belongs to the spectrum
is bipartite \cite[Proposition 3.4.1]{BrHa--12}.
This does not carry over to infinite graphs; indeed, there exists
a connected infinite graph which is not bipartite and of which the spectrum $\Sigma$
is symmetric with respect to $0$;
for example the second ladder-like graph in \cite[Example 4.1]{Goli--21}.
See however Proposition~\ref{countloops}~\ref{4DEcountloops}.

\begin{proof}[Proof of Proposition~\ref{1stqual}]
\ref{1DE1stqual}
We will use known results of the spectral theory of finite graphs.
\par

Let $V_F$ be a finite subset of $V$ such that
the finite subgraph $F = (V_F, E_F)$ of $G$
induced by $V_F$ is connected, and with $E_F \ne \emptyset$.
Let $A_F$ be the adjacency operator of $F$.
Write $M_F$ for $M_{A_F}$ and $m_F$ for $m_{A_F}$;
note that $M_F$ is the largest eigenvalue of $F$, that $m_F$ is the smallest,
and that $m_F \le -1 < 1 \le M_F$.
There exists a Perron--Frobenius eigenvector $\xi_F \in \ell^2(V_F)$ of $A_F$
such that $A_F \xi_F = M_F \xi_F$ and $\Vert \xi_F \Vert = 1$.
Let $\xi \in \ell^2(V)$ be defined by $\xi(v) = \xi_F (v)$ for $v \in V_F$
and $\xi(v) = 0$ for $v \in V \smallsetminus V_F$.
Then
$$
1 \le \Vert A_F \Vert = M_F = \langle A_F \xi_F \sca \xi_F \rangle
= \langle A \xi \sca \xi \rangle \le M_A \le \Vert A \Vert .
$$
Similarly, there exists an eigenvector $\eta_F \in \ell^2(V_F)$ of $A_F$
such that $A_F \eta_F = m_F \eta_F$ and $\Vert \eta_F \Vert = 1$.
Let $\eta \in \ell^2(V)$ be defined by $\eta(v) = \eta_F (v)$ for $v \in V_F$
and $\eta(v) = 0$ for $v \in V \smallsetminus V_F$.
Then 
$$
-\Vert A \Vert \le m_A \le \langle A \eta \sca \eta \rangle
= \langle A_F \eta_F \sca \eta_F \rangle = m_F \le -1 .
$$
For the equality $M_A = \Vert A \Vert$,
see Proposition~\ref{countloops}~\ref{3DEcountloops} below.

\vskip.2cm

\ref{2DE1stqual}
Let $V = V' \sqcup V''$ be a disjoint union such that
any edge in $G$ connects a vertex in $V'$ and a vertex in $V''$.
Define an operator $S$ on $\ell^2(V)$ by $(S\xi)(v) = \xi(v)$ for all $v \in V'$
and $(S\xi)(v) = -\xi(v)$ for all $v \in V''$.
Then $S$ is unitary and selfadjoint (hence $S^2 = {\rm id}$), and $SAS = -A$.
It follows that $\Sigma(A) = \Sigma(-A) = -\Sigma(A)$ is symmetric with respect to $0$.
\end{proof}
 
Let $v, w$ be two vertices of $G$.
We write $\mu_{v, w}$ instead of $\mu_{\delta_v, \delta_w}$
for the local spectral measure on $\Sigma (A)$ at $\delta_v$ and $\delta_w$;
we write $\mu_v$ instead of $\mu_{\delta_v}$
for the local spectral measure on $\Sigma (A)$ at $\delta_v$,
defined as in Formula~\eqref{EqLocalSpecMeas} of Section~\ref{SectionDomVectors}.
As in the introduction, we call this $\mu_v$
the \textbf{vertex spectral measure} of $A$ at $v$;
it is also called the \textbf{Kesten spectral measure} of $A$ at $v$,
see \cite{Kest--59}, \cite[Section 10]{Grig--11}.
The interest of the measures $\mu_{v, w}$ is that their moments
count walks in $G$ from $v$ to~$w$.
Note that $\mu_v$ is a positive measure;
when $v \ne w$, the measure $\mu_{v, w}$ is a real measure
(because its moments are real, see~\ref{1DEcountloops} below),
but not a positive measure (because $\int_{\Sigma (A)} d\mu_{v, w} (t) = 0$).
\par

We denote by $d(v,w)$ the combinatorial distance in $G$ between $v$ and $w$,
if $v$ and $w$ are in the same connected component of $G$;
otherwise set $d(v,w) = \infty$.

\begin{prop}[\textbf{moments of vertex spectral measures and numbers of walks}]
\label{countloops}
Let $G = (V, E)$, $A$, and $\Sigma (A)$ be as above.
We assume that $G$ is connected, and is not the trivial graph with one vertex and no edge.
\begin{enumerate}[label=(\roman*)]
\item\label{1DEcountloops}
Let $v, w \in V$, and $\mu_{v, w}$ be as above; let $n$ be a non-negative integer.
The moment of order $n$ of $\mu_{v, w}$ counts walks in $G$:
$$
\int_{\Sigma (A)} t^n d\mu_{v, w} (t) = \langle A^n \delta_v \sca \delta_w \rangle =
\left\{
\begin{aligned}
& \text{number of walks in $G$}
\\
& \text{of length $n$ from $v$ to $w$.}
\end{aligned}
\right.
$$
In particular $\int_{\Sigma (A)} t^n d\mu_v(t)$
is the number of closed walks in $G$ of length $n$ at $v$.
\item\label{2DEcountloops}
Let $\xi \in \ell^2(V)$ be such that $\Vert \xi \Vert = 1$,
$\xi (v) \ge 0$ for all $v \in V$,
and $\xi (v) = 0$ for all $v$ outside a finite subset $F$ of $V$.
Then $\lim_{n \to \infty} \langle A^{2n} \xi \sca \xi \rangle^{1/2n} = \Vert A \Vert$.
\item\label{3DEcountloops}
For $v \in V$, we have
$$
\lim_{n \to \infty} \langle A^{2n} \delta_v \sca \delta_v \rangle^{1/2n}
= \lim_{n \to \infty} \sqrt[\leftroot{-3}\uproot{3}2n]{
\begin{aligned}
&\text{number of closed walks}
\\
&\text{in $G$ of length $2n$ at $v$}
\end{aligned}
}
= \Vert A \Vert 
$$
and $\max \{ t \in \text{closed support of} \hskip.2cm \mu_v \} = \Vert A \Vert$.
\item\label{4DEcountloops}
Let $v, w \in V$ be such that $d(v,w)$ is even [respectively odd].
The graph $G$ is bipartite if and only if the measure $\mu_{v, w}$ is symmetric
[respectively antisymmetric] with respect to the origin.
\end{enumerate}
\end{prop}

\begin{proof}
\emph{Step one is the proof of Claim \ref{1DEcountloops}}.
This is straightforward,
using the definitions and the rule of multiplication of matrices.

\vskip.2cm 

\emph{Step two: for $\xi$ as in \ref{2DEcountloops},
the limit $\rho_\xi = \lim_{n \to \infty} \langle A^{2n+2} \xi \sca \xi \rangle$ exists,
and $\Vert A \xi \Vert \le \rho_\xi \le \Vert A \Vert$.}
\par

Let $n \in \N$.
Observe that $\langle A^{2n} \xi \sca \xi \rangle > 0$,
because there exist at least one closed walk of length $2n$ at each vertex of $G$.
By the Cauchy--Schwartz inequality, we have
$$
\langle A^{2n+2} \xi \sca \xi \rangle^2
= \langle A^{n+2} \xi \sca A^n \xi \rangle^2
\le \langle A^{n+2}\xi \sca A^{n+2} \xi \rangle
\langle A^n \xi \sca A^n \xi \rangle
$$
hence
$$
\frac{ \langle A^{2n+2} \xi \sca \xi \rangle }{ \langle A^{2n} \xi \sca \xi \rangle }
\le
\frac{ \langle A^{2n+4} \xi \sca \xi \rangle }{ \langle A^{2n+2} \xi \sca \xi \rangle } .
$$
It follows that these quotients have a limit
$$
\lim_{n \to \infty} \frac{ \langle A^{2n+2} \xi \sca \xi \rangle }{ \langle A^{2n} \xi \sca \xi \rangle }
=
\sup_{n \ge 0} \frac{ \langle A^{2n+2} \xi \sca \xi \rangle }{ \langle A^{2n} \xi \sca \xi \rangle } .
$$
We denote by $\rho_\xi$ the square root of this limit.
Since $ \langle A^0 \xi \sca \xi \rangle = 1$, we have
$$
\frac{ \langle A^2 \xi \sca \xi \rangle }{ \langle A^0\xi \sca \xi \rangle } \hskip.1cm
\frac{ \langle A^4 \xi \sca \xi \rangle }{ \langle A^2\xi \sca \xi \rangle } \hskip.1cm
\frac{ \langle A^6 \xi \sca \xi \rangle }{ \langle A^4\xi \sca \xi \rangle } \hskip.1cm
\cdots \hskip.1cm
\frac{ \langle A^{2n} \xi \sca \xi \rangle }{ \langle A^{2n-2}\xi \sca \xi \rangle }
=
\langle A^{2n} \xi \sca \xi \rangle
$$
and therefore
\small
$$
\lim_{n \to \infty}
\langle A^{2n} \xi \sca \xi \rangle^{1/2n}
= \sqrt{ \lim_{n \to \infty} \Big(
\frac{ \langle A^2 \xi \sca \xi \rangle }{ \langle A^0 \xi \sca \xi \rangle } \hskip.1cm
\frac{ \langle A^4 \xi \sca \xi \rangle }{ \langle A^2 \xi \sca \xi \rangle } \hskip.1cm
\frac{ \langle A^6 \xi \sca \xi \rangle }{ \langle A^4 \xi \sca \xi \rangle } \hskip.1cm
\cdots \hskip.1cm
\frac{ \langle A^{2n} \xi \sca \xi \rangle }{ \langle A^{2n-2}\xi \sca \xi \rangle } \Big)^{1/n}
}
= \rho_\xi .
$$
\normalsize
Note that
$$
\rho_\xi = \lim_{n \to \infty} \langle A^{2n} \xi \sca \xi \rangle^{1/2n}
= \sqrt{ \sup_{n \ge 0}
\frac{ \langle A^{2n+2} \xi \sca \xi \rangle }{ \langle A^{2n} \xi \sca \xi \rangle } }
\ge \sqrt{
\frac{ \langle A^2 \xi \sca \xi \rangle }{ \langle \xi \sca \xi \rangle } }
= \Vert A \xi \Vert .
\leqno{(\text{sup})}
$$
Since $\langle A^{2n} \xi \sca \xi \rangle \le \Vert A^{2n} \Vert \le \Vert A \Vert^{2n}$
for all $n \ge 0$, we have also $\rho_\xi \le \Vert A \Vert$.

\vskip.2cm 

\emph{Step three: there exists a number $\rho_G$
such that $0 \le \rho_G \le \Vert A \vert$
and $\rho_\xi = \rho_G$ for all $\xi$ as in \ref{2DEcountloops}}.
\par

When $\xi = \delta_v$, the number written $\rho_{\delta_v}$ in Step two
is written $\rho_{v,v}$ below.
Let $v, w \in V$; set $k = d(v,w)$.
On the one hand, $\langle A^n \delta_v \mid \delta_v \rangle$
is for each $n \ge 0$ the number of closed walks in $G$ of length $n$ at $v$,
and $\lim_{n \to \infty} \langle A^{2n} \delta_v \mid \delta_v \rangle^{1/2n} = \rho_{v,v}$.
On the other hand, for $n \ge k$, every closed walk of length $n-k$ at $v$ is the beginning of
a walk of length $n$ from $v$ to $w$, so that we have
$$
\langle A^{n-k} \delta_v \mid \delta_v \rangle
\le \langle A^n \delta_w \mid \delta_v \rangle ,
\hskip.5cm \text{and similarly} \hskip.5cm
\langle A^n \delta_w \mid \delta_v \rangle
\le \langle A^{n+k} \delta_v \mid \delta_v \rangle .
$$
From this and from (sup), it follows that there exists a limit
$$
\begin{aligned}
& \rho_{v,w} = \lim_{n \to \infty} \langle A^{2n} \delta_w \mid \delta_v \rangle^{1/2n}
\hskip.5cm \text{if $k$ is even}
\\
& \rho_{v,w} = \lim_{n \to \infty} \langle A^{2n+1} \delta_w \mid \delta_v \rangle^{1/(2n+1)}
\hskip.5cm \text{if $k$ is odd}
\end{aligned}
$$
and that $\rho_{v,w} = \rho_{v,v}$.
Since walks from $v$ to $w$ are in bijection with walks from $w$ to $v$,
we have moreover $\rho_{v,w} = \rho_{w,v}$.
It follows that $\rho_{v,w}$ is independent of the choices of $v$ and $w$;
from now on, we write $\rho_G$ for this number.
Since $ \langle A^{2n+1} \delta_w \mid \delta_v \rangle
\le \sum_j \langle A^{2n} \delta_w \mid \delta_{v_j} \rangle$,
where the summation is over the neighbours $v_j$ of $v$ in $G$,
we have also
$$
\limsup_{n \to \infty} \langle A^n \delta_w \mid \delta_v \rangle^{1/n} = \rho_G
\hskip.5cm \text{for all} \hskip.2cm
v, w \in V .
\leqno{(\text{rho1})}
$$
Note that, for~\ref{3DEcountloops},
we still have to show that $\rho_G \ge \Vert A \Vert$.
\par

Let now $\xi$ and $F$ be as in~\ref{2DEcountloops}.
We have $\xi = \sum_{v \in F} \xi(v) \delta_v$ and
$$
\langle A^n \xi \sca \xi \rangle = \sum_{v,w \in F} \xi(v) \xi(w) \langle A^n \delta_w \sca \delta_v \rangle .
$$
Whenever $\xi(v) \xi(w) > 0$, we have
$$
\limsup_{n \to \infty} \Big( \xi(v) \xi(w) \langle A^n \delta_w \sca \delta_v \rangle \Big)^{1/n} = \rho_G
$$
by (rho1), and it follows that
$$
\begin{aligned}
\rho_\xi
&= \limsup_{n \to \infty} \langle A^n \xi \sca \xi \rangle^{1/n}
\\
& = \limsup_{n \to \infty} \bigg(
\sum_{v,w \in F} \xi(v) \xi(w) \langle A^n \delta_w \sca \delta_v \rangle
\bigg)^{1/n}
= \rho_G .
\end{aligned}
\leqno{(\text{rho2})}
$$

\vskip.2cm

\emph{Step four: we have $\rho_G = \Vert A \vert$
and 
$$
\max \{ t \in \text{closed support of} \hskip.2cm \mu_w \} = \Vert A \Vert 
$$
for all $w \in V$.}
\par

Choose $\varepsilon > 0$.
There exists $\eta \in \ell^2(V)$ such that $\Vert \eta \Vert = 1$,
and $\eta(v) = 0$ for all $v$ outside a finite subset $F$ of $V$,
and $\Vert A \eta \Vert \ge \Vert A \Vert - \varepsilon$.
Let $\xi \in \ell^2(V)$ be defined by $\xi (v) = \vert \eta (v) \vert$ for all $v \in V$.
Then $\Vert \xi \Vert = 1$,
and $\xi(v) = 0$ for all $v$ outside $F$,
and $\langle A^2 \xi \sca \xi \rangle \ge \langle A^2 \eta \sca \eta \rangle$.
It follows from (sup) and (rho2) above that
$$
\rho_G \ge \Vert A \xi \Vert \ge \Vert A \eta \Vert \ge \Vert A \Vert - \varepsilon ,
$$
and this shows that $\rho_G = \Vert A \Vert$.
\par

Let $\Sigma_w$ denote the closed support of $\mu_w$;
set $M_w = \max \{ t \in \Sigma_w \}$.
Choose $\varepsilon > 0$ and let $c$ be the positive number
$\mu_w \left( \mathopen[ M_w - \varepsilon, M_w \mathclose] \right)$. Then
$$
\begin{aligned}
\Vert A \Vert
& = \lim_{n \to \infty} \langle A^{2n} \delta_w \sca \delta_w \rangle^{1/2n}
= \lim_{n \to \infty} \bigg( \Big( \int_{\Sigma_w} t^{2n} d\mu_w (t) \Big)^{1/2n} \bigg)
\\
& \ge \lim_{n \to \infty}
\bigg( \Big( \int_{M_w - \varepsilon}^{M_w} t^{2n} d\mu_w (t) \Big)^{1/2n} \bigg)
\ge \lim_{n \to \infty} \bigg( \Big( c (M_w - \varepsilon)^{2n} \Big)^{1/2n} \bigg)
= M_w - \varepsilon ,
\end{aligned}
$$
hence $\Vert A \Vert = M_w$, and this ends the proof.

\vskip.2cm

For Claim~\ref{4DEcountloops}, we refer to \cite[Theorem 2.1]{GoMo--88}.
\end{proof}


\begin{defn}
\label{defdomv}
In a graph $G = (V, E)$,
a vertex $v \in V$ is
\textbf{dominant} if the vector $\delta_v \in \ell^2(V)$ is dominant for~$A$,
and \textbf{cyclic} if $\delta_v$ is cyclic for~$A$.
\par

By Proposition~\ref{dom=dom}, a vertex $v$ is dominant in $G$
as soon as the vertex spectral measure $\mu_v$ dominates $\mu_w$
for all $w \in V$.
\end{defn}

As examples below will show, the vertices of a graph need not be all dominant,
indeed some graphs have no dominant vertex at all.
In other terms, a vertex spectral measure need not be
a scalar-valued spectral measure.
\par

We provide in this section examples of graphs for which
it is easy to check that all vertices are dominant.
There are further examples in Section~\ref{SectionFE},
for which the determination of dominant vertices is not so straightforward,
and relies on the analysis of a family of Jacobi operator in Section~\ref{S+EofJc}.

\begin{prop}[\textbf{all vertices are dominant in vertex-transitive graphs}]
\label{transimpliesdom}
Let $G = (V, E)$ be a graph of bounded degree and $A$ its adjacency operator.
\begin{enumerate}[label=(\arabic*)]
\item\label{1DEtransimpliesdom}
If $v, w \in V$ are two vertices such that there exists an automorphism of $G$
mapping $v$ to $w$, the vertex spectral measures $\mu_v$ and $\mu_w$ are equal.
\item\label{2DEtransimpliesdom}
If the group of automorphisms of $G$ is transitive on $V$,
then all vertex spectral measures are equal, and all vertices are dominant.
\end{enumerate}
\end{prop}

\begin{proof}
\emph{Preliminary step.}
An automorphism $\alpha$ of $G$ induces a unitary operator $U_\alpha$ on $\ell^2(V)$,
defined by $U_\alpha \delta_v = \delta_{\alpha(v)}$ for all $v \in V$.
The operator $U_\alpha$ commutes with $A$;
indeed, for all $v, w \in V$,
$$
\langle U_\alpha^* A U_\alpha \delta_v \sca \delta_w \rangle =
\langle A U_\alpha \delta_v \sca U_\alpha \delta_w \rangle =
\langle A \delta_{\alpha(v)} \sca \delta_{\alpha(w)} \rangle =
\left\{
\begin{aligned}
1 \hskip.2cm & \hskip.2cm
\text{if} \hskip.2cm \{\alpha(v), \alpha(w)\} \in E
\\
0 \hskip.2cm & \hskip.2cm
\text{if not,}
\end{aligned}
\right.
$$
so that $\langle U_\alpha^* A U_\alpha \delta_v \sca \delta_w \rangle
= \langle A \delta_v \sca \delta_w \rangle$;
it follows that $U_\alpha^* A U_\alpha = A$, i.e., that $U_\alpha$ commutes with $A$.
This implies in turn that $U_\alpha$ commutes with $f(A)$
for any function $f \in \mathcal B (\Sigma (A))$.

\vskip.2cm

\emph{Proof of~\ref{1DEtransimpliesdom}.}
Let $\alpha$ be an automorphism of $G$ such that $\alpha (v) = w$.
For all $n \ge 0$, we have
$$
\langle A^n \delta_w \sca \delta_w \rangle
= \langle A^n U_\alpha \delta_v \sca U_\alpha \delta_v \rangle
= \langle U_\alpha A^n \delta_v \sca U_\alpha \delta_v \rangle
= \langle A^n \delta_v \sca \delta_v \rangle
$$
and therefore $\int_{\Sigma (A)} t^n d\mu_w(t) = \int_{\Sigma (A)} t^n d\mu_v(t)$.
This implies that the measures $\mu_w$ and $\mu_v$ are equal.

\vskip.2cm

\emph{Proof of~\ref{2DEtransimpliesdom}.}
If $G$ is vertex-transitive,
all measures $\delta_v$ coincide,
and it follows from Proposition~\ref{dom=dom}
that $\delta_v$ is dominant for all $v \in V$.
\end{proof}

\begin{exe}[\textbf{all vertices are dominant in Cayley graphs,
in particular in grids and regular trees}]
\label{exLatticeTree}
In a Cayley graph, all vertices are dominant,
by Item~\ref{2DEtransimpliesdom} of Proposition~\ref{transimpliesdom}.
There are numerous families of well-studied graphs which are Cayley graphs:
\par

For example,
connected circulant graphs, ladder graphs, and Paley graphs
are finite Cayley graphs
(see for example Chapter 17 in \cite{Bigg--93}).
For any positive integer $d$, 
the $d$-dimensional square lattice $L^d$ is the Cayley graph
of the free abelian group $\Z^d$ with a free basis as generating set;
note that $L^1$ is the infinite line.
For $d \ge 3$ the regular tree $T^d$ of degree $d$ is the Cayley graph
of the free product of $d$ copies of the group of order two
with the canonical generating set of $d$ elements.
(The tree $T^2$ makes sense, but it is the same graph as the line $L^1$,
also known as the two-way infinite path.)
In all these graphs, all vertices are dominant.
\end{exe}

\begin{defn}
\label{defwalkr}
A graph $G$ is \textbf{walk-regular}
if the number of closed walks of length $n$ at a vertex $v$ of $G$
is independent of $v$ for all $n \ge 0$,
equivalently
if the measure~$\mu_v$ is independent on the vertex~$v$,
equivalently
if for all $k \ge 0$ the diagonal matrix coefficients of the power $A^k$
of the adjacency operator $A$ of $G$ are equal
(the terminology is that of \cite{GoMc--80}).
\end{defn}

Vertex-transitive graphs are walk-regular.
Distance-regular graphs (definition recalled in Example~\ref{DS}) are walk-regular.
Examples in \cite{GoMc--80} show that
a walk-regular graph need be neither vertex-transitive nor distance-regular.

\begin{prop}[\textbf{all vertices are dominant in walk-regular graphs}]
\label{wrg}
Let $G = (V, E)$ be a graph which is walk-regular.
Then any vertex $v \in V$ is dominant.
\end{prop}

\begin{proof}
For $G$ a walk-regular graph,
the moments of the measures $\mu_v$ do not depend on $v$
(see Proposition~\ref{countloops}),
hence the measures $\mu_v$ themselves do not depend on $v$,
hence all vertices are dominant by Proposition~\ref{dom=dom}.
\end{proof}

\begin{comp}
\label{compilation}
Here is a list of the various ``spectral measures'' which appear in this context.
For a bounded self-adjoint operator $X$ on a Hilbert space $\Hi$,
with spectrum $\Sigma$, the first three terms below
are defined in Section~\ref{SectionDomVectors}
and the fourth appears in Section~\ref{SectionFG}:
\begin{enumerate}[label=(\arabic*)]
\item\label{1DEcompilation}
The operator-valued \textbf{spectral measure} $E_X$,
a projection-valued measure such that $X = \int_\Sigma t dE_X(t)$.
\item\label{2DEcompilation}
The \textbf{scalar-valued spectral measures},
which are the finite positive measures on $\Sigma$
having the same Borel subsets of $\Sigma$ of measure zero as $E_X$.
\item\label{3DEcompilation}
The \textbf{local spectral measure} of $X$ at a vector $\xi \in \Hi$,
which is a scalar-valued spectral measure if and only if the vector $\xi$ is dominant.
\item\label{4DEcompilation}
When $\Hi$ is finite dimensional, and $X$ has eigenvalues $\lambda_1, \hdots, \lambda_s$
with multiplicities $m_1, \hdots, m_s$ respectively,
the \textbf{counting measure} $\tau$ is defined by
$\tau (\{\lambda_j\}) = \frac{m_j}{\dim \Hi}$;
it is a scalar-valued spectral measure on $\Sigma$.
\end{enumerate}
For the adjacency operator $A$ of a graph $G = (V, E)$ of bounded degree,
which is a bounded self-adjoint operator on $\ell^2(V)$, we have:
\begin{enumerate}[label=(\arabic*)]
\addtocounter{enumi}{4}
\item\label{5DEcompilation}
The \textbf{vertex spectral measure} $\mu_v$ at $v$
also called the \textbf{Kesten spectral measure} at $v$,
which is the local spectral measure of $A$ at the vector $\delta_v \in \ell^2(V)$,
and which is a scalar-valued spectral measure if and only the vertex $v$ is dominant
(see just before Proposition~\ref{countloops}).
When the graph $G$ is walk-regular
(and therefore regular of some degree $d$),
the measure $\mu_v$ is independent of $v$
and it is sometimes called the \textbf{Plancherel measure}
\cite[\S~3.4]{FiPi--83}, \cite[Page 219]{MoWo--89}.
\item\label{6DEcompilation}
When $G$ is finite, so that $A$ is as in~\ref{4DEcompilation} above,
the counting measure on the spectrum of $A$ is also called the \textbf{density of states},
as noted below, before Lemma~\ref{lemcountingmeasure}.
\end{enumerate}
\end{comp}

\section{The case of finite graphs}
\label{SectionFG}

We first recall some facts from linear algebra.
Consider a linear endomorphism $X$
of a finite dimensional vector space $\Hi$ of positive dimension.
For the next three definitions and for Proposition~\ref{proplinalg},
$X$ need not be self-adjoint,
indeed $\Hi$ can be a vector space over any field, say $\K$.
Let $\K [T]$ denote the ring of polynomial in one variable over~$\K$.
The \textbf{minimal polynomial} $P_{{\rm min}} \in \K [T]$ of $X$
is the monic polynomial which generates the ideal $\{ Q \in \K[T] : Q(X) = 0 \}$,
equivalently the monic polynomial of minimal degree such that $P_{{\rm min}} (X) = 0$.
For $\xi \in \Hi$,
the \textbf{minimal polynomial $P_\xi \in \K [T]$ of $(X, \xi)$}
is the monic polynomial which generates the ideal
$\{ Q \in \K[T] : Q(X)\xi = 0 \}$,
equivalently the monic polynomial of minimal degree such that $P_\xi (X)\xi = 0$.
Observe that $P_\xi$ divides $P_{{\rm min}}$.
A vector $\xi \in \Hi$ is \textbf{dominant for $X$} if, for $Q \in \K[T]$,
the equality $Q(X) \xi = 0$ implies $Q(X) = 0$.

\begin{prop}
\label{proplinalg}
Let $X$ be an endomorphism of a finite dimensional vector space~$\Hi$
over a field $\K$.
Let $P_{{\rm min}}$ be the minimal polynomial of $X$,
and let $s$ be its degree.
\begin{enumerate}[label=(\arabic*)]
\item\label{1alaGant}
There exist dominant vectors for $X$.
\item\label{2alaGant}
Let $\xi \in \Hi$ and let $P_\xi$ be the minimal polynomial of $(X, \xi)$.
The following conditions are equivalent:
\begin{enumerate}[label=(\roman*)]
\item\label{iDEproplinalg}
$\xi$ is dominant for $X$,
\item\label{iiDEproplinalg}
$P_\xi = P_{{\rm min}}$,
i.e.,
the degree of $P_\xi$ is $s$,
\item\label{iiiDEproplinalg}
the vectors $\xi, X\xi, X^2\xi, \hdots, X^{s-1}\xi$ are linearly independent.
\end{enumerate}
\end{enumerate}
\end{prop}

\begin{proof}
\ref{1alaGant}
Let $P_{{\rm min}} = \prod_{j=1}^d (Q_j)^{\ell_j}$
be the decomposition of the minimal polynomial
in product of powers of irreducible polynomials in $\K [T]$.
There exists a direct sum decomposition $\Hi = \bigoplus_{j=1}^d V_j$
in $X$-invariant subspaces such that, for each $j \in \{1, \hdots, d\}$,
the minimal polynomial of the restriction of $X$ to $V_j$ is $Q_j^{\ell_j}$.
For appropriate choices of vectors $\xi_j \in V_j$,
the sum $\sum_{j=1}^d \xi_j$ is dominant for $X$.
For the details, see ~\cite[\S~VII.2]{Gant--77}.
\par

Note that Gantmacher does not have a specific terminology
for the vectors which are called dominant here.
\par

In the particular case of a diagonalisable operator,
the $V_j$~'s are the eigenspaces of $X$
and $\sum_{j=1}^d \xi_j$ is dominant for $X$
for any choice of nonzero vectors $\xi_j \in V_j$.

\vskip.2cm

\ref{2alaGant}
It follows from the definitions that~\ref{iDEproplinalg} holds, i.e.,
that the ideals
$\{ Q \in \K[T] : Q(X) = 0 \}$ and $\{ Q \in \K[T] : Q(X) \xi = 0 \}$ coincide,
if and only if $P_\xi = P_{{\rm min}}$, i.e.,
if and only if~\ref{iiDEproplinalg} holds.
Also~\ref{iiDEproplinalg} holds, i.e.,
$Q(X)\xi \ne 0$ for any polynomial $Q \ne 0$ of degree at most $s-1$,
if and only if $\xi, X\xi, X^2\xi, \hdots, X^{s-1}\xi$ are linearly independent, i.e.,
if and only if~\ref{iiiDEproplinalg} holds.
\end{proof}

We return now to the situation of a self-adjoint operator $X$
acting on a finite dimensional Hilbert space $\Hi$.
A vector $\xi \in \Hi$ is dominant in the sense of Proposition~\ref{proplinalg}
if and only if
it is dominant in the sense of Section~\ref{SectionDomVectors},
see~\ref{2DEdominant} in Proposition~\ref{dominant}.
It is convenient to arrange the eigenvalues which constitute the spectrum $\Sigma (X)$
in decreasing order:
$$
\lambda_1 > \lambda_2 > \cdots > \lambda_s .
$$
For $j \in \{1, \hdots, s\}$, denote
by $E_j$ the eigenspace $\ker (\lambda_j - X)$.
The space $\Hi$ is the orthogonal direct sum
of the $X$-invariant eigenspaces:
$\Hi = \bigoplus_{j=1}^s E_j$.
For $\xi \in \Hi$, let $\xi_j$ be the orthogonal projection of $\xi$ in $E_j$,
so that $\xi = \sum_{j=1}^s \xi_j$ and $X \xi = \sum_{j=1}^s \lambda_j \xi_j$.
Since $\Sigma (X)$ is finite, the local spectral measure $\mu_\xi$
is a linear combination of the Dirac measures at the $\lambda_j$~'s.
The moments of $\mu_\xi$ of $X$ at $\xi$ are
\begin{equation}
\label{EqMomentMuxi}
\int_{\Sigma (X)} t^n d\mu_\xi(t) = \langle X^n \xi \sca \xi \rangle
= \bigg\langle \sum_{j=1}^s \lambda_j^n \xi_j 
\hskip.1cm \bigg\vert \hskip.1cm
\sum_{k=1}^s \xi_k \bigg\rangle
= \sum_{j=1}^s \lambda_j^n \Vert \xi_j \Vert^2 
\end{equation}
for all $n \ge 0$.
\par

Let $P_j \, \colon \Sigma (X) \to \C$ be the function
such that $P_j(\lambda_j) = 1$ and $P_j(\lambda_k) = 0$ for $k \ne j$,
equivalently the restriction to $\Sigma (X)$ of the polynomial defined by
$P_j(\lambda) = \big( \prod_{k \ne i} (\lambda_j - \lambda_k) \big)^{-1}
\prod_{k \ne j} (\lambda - \lambda_k)$.
Then
$$
\mu_\xi \big( \{ \lambda_j \} \big)
= \langle P_j(X) \xi \sca \xi \rangle = \langle \xi_j \sca \xi \rangle
= \Vert \xi_j \Vert^2 
\hskip.5cm
\text{which is $>0$ if and only if $\xi_j \ne 0$.}
$$
This shows that Conditions~\ref{ivDEproplinalgsa} and~\ref{vDEproplinalgsa}
of the following proposition are equivalent.

\begin{prop}
\label{proplinalgsa}
Let $X$ be a self-adjoint operator on a finite-dimensional Hilbert space $\Hi$.
Let $\xi \in \Hi$ and let $\mu_\xi$ be the local spectral measure of $X$ at $\xi$.
The condition for $\xi$ to be dominant
(Condition~\ref{iDEproplinalg} of Proposition~\ref{proplinalg})
is moreover equivalent to each of
\begin{enumerate}[label=(\roman*)]
\addtocounter{enumi}{3}
\item\label{ivDEproplinalgsa}
the orthogonal projection of $\xi$ on each eigenspace of $X$ is not zero,
\item\label{vDEproplinalgsa}
$\mu_\xi( \{\lambda\} ) > 0$ for every eigenvalue $\lambda$ of $X$.
\end{enumerate}
\end{prop}

\begin{proof}
It suffices now to check that
one of the conditions above is equivalent
to one of the conditions of Proposition~\ref{proplinalg}.
This holds because Condition~\ref{vDEproplinalgsa} holds
if and only if $\xi$ is dominant,
see Proposition~\ref{dominant}~\ref{3DEdominant}.
\end{proof}

Consider again the adjacency operator of a finite graph.
Proposition~\ref{localspdominantvertex}
is a complement to Proposition~\ref{countloops}
and Proposition~\ref{domvertfg}
repeats for graphs the equivalent conditions of
Propositions~\ref{proplinalg} and ~\ref{proplinalgsa}.

\begin{prop}[\textbf{counting walks in terms of eigenvalues}]
\label{localspdominantvertex}
Let $G = (V, E)$ be a finite graph,
$\lambda_1 > \lambda_2 > \cdots > \lambda_s$
the eigenvalues of its adjacency operator~$A$,
and $v \in V$ a vertex of $G$.
For all $n \ge 0$, the number of closed walks in $G$ of length~$n$
starting at a vertex $v \in V$ and ending at $v$ is
\begin{equation}
\label{EqAnvv}
\langle A^n \delta_v \sca \delta_v \rangle
= \sum_{j=1}^s c_j^2 \lambda_j^n ,
\end{equation}
where the $c_j$~'s are non-negative numbers,
more precisely $c_j$ is the norm of the orthogonal projection of $\delta_v$
on the eigenspace of $A$ of eigenvalue $\lambda_j$.
\par

More generally, for two vertices $v$ and $w$,
there exist constants $d_1, \hdots, d_s \in \R$ such that
the number of walks in $G$ of length $n$ starting at $v$ and ending at $w$ is
\begin{equation}
\label{EqAnvw}
\langle A^n \delta_v \sca \delta_w \rangle = \sum_{j=1}^s d_j \lambda_j^n .
\end{equation}
\end{prop}

\begin{proof}
For $j \in \{1, \hdots, s\}$, denote by $P_j$ the orthogonal projection
of $\ell^2(V)$ onto the eigenspace of $A$ of eigenvalue $\lambda_j$.
For $n \ge 0$ and $v \in V$, we have
$$
\langle A^n \delta_v \sca \delta_v \rangle
= \langle \sum_{j=1}^s \lambda_j^n P_j \delta_v \sca \delta_v \rangle
= \sum_{j=1}^s \langle P_j \delta_v \sca \delta_v \rangle \lambda_j^n ,
$$
and (4) follows. The proof of (5) is similar.
\end{proof}

Formulas~\eqref{EqAnvv} and~\eqref{EqAnvw} are well-known;
see for example \cite[Formulas~2.2.1 and~4.2.5]{CvRS--97}.
The number $c_j$ is $\cos( {\rm angle} (\delta_v, E_j) )$, and $c_j^2 \ge 0$.
Since $\Vert \delta_v \Vert = 1$, we have $\sum_{j=1}^s c_j^2 = 1$.
When $G$ is connected, note also that $c_1^2 > 0$,
because $\lambda_1$ has a Perron--Frobenius eigenvector,
of which all coordinates
with respect to the basis $(\delta_v)_{v \in V}$ are positive.
The number
$d_j = \cos( {\rm angle} (\delta_v, E_j) ) \cos( {\rm angle} (\delta_w, E_j) )$
can be positive, zero, or negative
(all cases occur in formulas for the path $P_3$).

\begin{prop}[\textbf{characterizations of dominant vertices in finite graphs}]
\label{domvertfg}
Let $G = (V, E)$ be a finite graph, $A$ its adjacency operator,
and $\ell^2(V) = \bigoplus_{j=1}^s E_j$ the orthogonal decomposition of $\ell^2(V)$
in eigenspaces of $A$.
For a vertex $v \in V$, the following properties are equivalent:
\begin{enumerate}[label=(\roman*)]
\item\label{iDEdomvertfg}
$v$ is a dominant vertex for $A$,
\item\label{iiiDEdomvertfg}
the minimal polynomial of $(A, \delta_v)$ is equal to the minimal polynomial
$\prod_{j=1}^s (T - \lambda_j)$ of $A$, 
\item\label{iiDEdomvertfg}
the vectors $\delta_v, A \delta_v, A^2 \delta_v, \hdots, A^{s-1}\delta_v$ in $\ell^2(V)$
are linearly independent,
\item\label{ivDEdomvertfg}
the orthogonal projection of $\delta_v$ on $E_j$ is not zero for all $j = 1, \hdots, s$,
\newline
i.e., $c_j^2 > 0$ for all $j = 1, \hdots, s$ in
Formula~\eqref{EqAnvv}
of Proposition~\ref{localspdominantvertex},
\item\label{vDEdomvertfg}
$\mu_v( \{\lambda\} ) > 0$ for every eigenvalue $\lambda$ of $A$.
\end{enumerate}
\end{prop}

Property~\ref{ivDEdomvertfg} of Proposition~\ref{domvertfg}
and Formula~\eqref{EqAnvv} above show that
the numbers of closed walks starting and ending at a dominant vertex
involve powers of all eigenvalues,
while the corresponding numbers at a non-dominant vertex
involve powers of some eigenvalues only.
However the difference is rather subtle.
For example, for $G = P_3$ the path with two end vertices $v, v'$, which are dominant,
and one middle vertex $w$, which is not dominant, we have
$$
\begin{matrix}
\langle A^{2m+1} \delta_v \sca \delta_v \rangle = 0
& \hskip.3cm \text{and} \hskip.3cm 
& \langle A^{2m+2} \delta_v \sca \delta_v \rangle = 2^m
& \hskip.3cm \text{for all} \hskip.3cm 
& m \ge 1 ,
\\
\langle A^{2m+1} \delta_w \sca \delta_w \rangle = 0
& \hskip.3cm \text{and} \hskip.3cm 
& \langle A^{2m+2} \delta_w \sca \delta_w \rangle = 2^{m+1}
& \hskip.3cm \text{for all} \hskip.3cm 
& m \ge 1 .
\end{matrix}
$$
More on the graphs $P_3$, and $P_n$ for all $n \ge 2$,
in the upcoming Example~\ref{finitepaths}.

\vskip.2cm

The next corollary shows the relationship between our terminology of dominant vertices
and that of null vertices of other authors.
Let $\xi \in \ell^2(V)$ be an eigenvector of $A$ of some eigenvalue $\lambda$.
A vertex $v \in V$ is a \textbf{null vertex for the vector} $\xi$ if $\xi(v) = 0$.
A vertex $v \in V$ is a \textbf{null vertex for an eigenvalue} $\lambda$
if it is a null vertex for any eigenvector of eigenvalue~$\lambda$.
Null vertices occur in the description of nodal domains for eigenvectors of $A$,
see~\cite{BiLS--07}. 
The terminology is also used in~\cite{AkGM--13}.

\begin{cor}[\textbf{finite graphs without dominant vertices}]
\label{nullvectorsonly}
Let $G = (V, E)$ be a finite graph, $v \in V$, and $A$ the adjacency operator of $G$.
Then
\begin{enumerate}[label=(\roman*)]
\item\label{iDEnullvectorsonly}
$v$ is not dominant if and only if $v$ is a null vertex
for some eigenvalue of~$A$.
\end{enumerate}
The following two properties are equivalent:
\begin{enumerate}[label=(\roman*)]
\addtocounter{enumi}{1}
\item\label{iiDEnullvectorsonly}
the graph $G$ does not have any dominant vertex,
\item\label{iiiDEnullvectorsonly}
any vertex in $V$ is a null vertex
for at least one eigenvalue of $A$.
\end{enumerate}
\end{cor}

Let $\Hi$ be a finite dimensional complex vector space,
with a basis $\varepsilon_* := (\varepsilon_j)_{j=1}^n$,
and let $X$ be a linear operator on $\Hi$
of which the matrix with respect to $\varepsilon_*$ has its entries in $\Q$.
Then $X$ can also be seen as a linear operator
on the $\Q$-vector space $\Hi_\Q := \bigoplus_{j=1}^n \Q \varepsilon_j$,
so that the characteristic polynomial of $X$,
the minimal polynomial of $X$,
and the minimal polynomial of $X$ at any vector in $\Hi_\Q$,
are polynomials in $\Q [T]$.
This applies in particular to the adjacency operator $A$ of a finite graph,
so that the \textbf{characteristic polynomial of the graph} (i.e., of $A$)
is a polynomial in $\Q [T]$;
as such, it can be reducible or irreducible.

\begin{prop}[\textbf{finite graphs with characteristic polynomials irredu\-cible over~$\Q$}]
\label{PropIrrQ}
Let $G$ be a finite graph of which the characteristic polynomial is irreducible over $\Q$.
Then all vertices of $G$ are dominant.
\end{prop}

\begin{proof}
Let $A$ be the adjacency operator of $G$
and $P_A$ its characteristic polynomial, viewed in $\Q [T]$.
Let $v$ be a vertex of the graph
and $P_v$ the minimal polynomial of the pair $(A, \delta_v)$.
Then $P_v$ divides $P_A$, and $P_v$ is not a constant polynomial.
If $P_A$ is irreducible, $P_v = P_A$
and therefore $v$ is dominant,
by Proposition~\ref{proplinalg}
(Note that, when $P_A$ is irreducible, it coincides with the minimal polynomial of $A$.)
\end{proof}

\begin{exe}[\textbf{irreducibility of characteristic polynomial
of Coxeter graph of type $E_8$}]
\label{exE8}
In the Coxeter graph of type $E_8$ (see Figure~1),
all vertices are dominant, because
its characteristic polynomial $P(T) = T^8 - 7 T^6 + 14 T^4 - 8 T^2 + 1$
is irreducible over~$\Q$.

\[\begin{tikzpicture}
\vertex[fill] (a) at (0,1) {};
\vertex[fill] (b) at (1,1) {};
\vertex[fill] (r) at (2,1) {};
\vertex[fill] (h) at (2,2) {};
\vertex[fill] (x) at (3,1) {};
\vertex[fill] (y) at (4,1) {};
\vertex[fill] (z) at (5,1) {};
\vertex[fill] (t) at (6,1) {};	\path
(a) edge (b)
(b) edge (r)
(r) edge (x)
(r) edge (h)
(x) edge (y)
(y) edge (z)
(z) edge (t) ;
\end{tikzpicture}\]
\vskip.1cm
\begin{center}
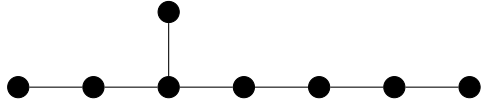
Figure 1. The graph $E_8$ of Example~\ref{exE8}
\end{center}
\vskip.2cm

Let us check that $P(T)$ is irreducible.
For a positive integer $n$, let $\Phi_n (T)$ denote the cyclotomic polynomial
with roots the $n^{{\rm th}}$ primitive roots of unity.
We have $\Phi_{60}(T) = \Phi_{30}(T^2) = T^{16} + T^{14} - T^{10} - T^8 -T^6 + T^2 + 1$.
The straightforward computation
$$
\begin{aligned}
& P \left( T + T^{-1} \right)
\\
& \hskip.2cm
\begin{matrix}
= & T^8 & +8 T^6 & +28T^4 & + 56 T^2 & + 70 & + 56 T^{-2} & + 28 T^{-4} & + 8 T^{-6} & + T^{-8}
\\
&& - 7 T^6 & -42 T^4 & - 105 T^2 & - 140 & -105 T^{-2} & - 42 T^{-4} & - 7 T^{-6} &
\\
&&& + 14 T^4 & + 56 T^2 & + 84 & + 56 T^{-2} & + 14 T^{-4} &&
\\
&&&& -8 T^2 & -16 & -8 T^{-2} &&&
\\
&&&&& + 1 &&&&
\\
= & T^8 & + T^6 && -T^2 & - 1 & - T^{-2} && + T^{-6} & + T^{-8} 
\end{matrix}
\end{aligned}
$$
shows that
$P \left( T + T^{-1} \right) = T^{-8} \Phi_{60} (T)$.
If $P (T)$ were reducible, $\Phi_{60}$ would also be reducible,
but this is not by a theorem of Gauss. It follows that $P(T)$ is irreducible.
\par

The characteristic polynomials of other Coxeter graphs are reducible.
See Examples~\ref{exCoxFinite} to~\ref{exCoxAffine} below.
\end{exe}

This and other examples of finite graphs
with characteristic polynomial irreducible over $\Q$
are shown in \cite{YLZH--21} and \cite{LiSi--22}.
Recall from the introduction that, conjecturally,
``most graphs'' have characteristic polynomial irreducible over $\Q$.

\vskip.2cm

Let $\Hi$ be a finite dimensional Hilbert space
and $X$ a self-adjoint operator on~$\Hi$.
The \textbf{counting measure} for $X$
is the probability measure $\tau$ on $\Sigma (X)$ defined by
$$
\tau (\{\lambda\})
= \frac{\text{multiplicity of} \hskip.2cm \lambda}{\dim \Hi}
\hskip.5cm \text{for each} \hskip.2cm
\lambda \in \Sigma (X) .
$$
The terminology ``counting measure'' is that of \cite{Grig--11};
other authors use \textbf{density of states}, or DOS,
a term which refers to solid state physics.
Note that $\tau$ is in particular a scalar-valued spectral measure for $X$.

\begin{lem}
\label{lemcountingmeasure}
Let $X$ be a self-adjoint operator on a Hilbert space $\Hi$
of finite dimension~$n$.
Let $(\varepsilon_i)_{1 \le i \le n}$ be an orthonormal basis of $\Hi$.
For $i \in \{1, \hdots, n\}$, denote by~$\mu_i$
the local spectral measure of $X$ at $\varepsilon_i$.
\par

Then $\tau = \frac{1}{n} \sum_{i=1}^n \mu_i$ is the counting measure for $X$.
\end{lem}

\begin{proof}
Let $\mathcal L (\Hi)$ denote the algebra of all operators on $\Hi$;
it is isomorphic to the matrix algebra ${\rm M}_n(\C)$.
Let $\C[X]$ be the subalgebra of $\mathcal L (\Hi)$ generated by~$X$;
any $Y \in \C[X]$ is of the form $f(X)$ for some function $f \, \colon \Sigma (X) \to \C$.
For $i \in \{1, \hdots, n\}$, the measure $\mu_i$,
viewed as a linear form on $\C[X]$, extends to the linear form
$$
\mathcal L (\Hi) \to \C ,
\hskip.2cm
Y \mapsto \langle Y \varepsilon_i \sca \varepsilon_i \rangle ,
\hskip.5cm
\text{again denoted by $\mu_i$}.
$$
The linear form $\tau = \frac{1}{n} \sum_{i=1}^n\mu_i$
is the normalized trace on $\mathcal L (\Hi)$.
\par

Denote by
$\lambda_1, \lambda_2, \hdots, \lambda_s$ the eigenvalues $X$
and by $m_1, m_2, \hdots, m_s$ their multiplicities.
For $j \in \{1, \hdots s\}$, let $P_j$ denote the orthogonal projection
of $\Hi$ onto the eigenspace $\ker( \lambda_j - X)$.
Let $Y \in \C[X]$;
let $f$ be a function $\Sigma (X) \to \C$ such that $Y = f(X)$;
we have
$$
\tau(Y) = \tau(f(X)) = \tau \Big( \sum_{j=1}^s f(\lambda_j) P_j \Big)
= \frac{1}{n} \sum_{j=1}^s f(\lambda_j) m_j .
$$
In the particular case of a characteristic function of the eigenvalue $\lambda_j$ of $X$,
for some $j \in \{1,\hdots, s\}$, this becomes
$\tau (\{\lambda_j\}) = \frac{m_j}{n} > 0$;
and this ends the proof.
\end{proof}

As a particular case, we obtain the following fact which is well-known,
and explicit for example in \cite[Proposition 10.2]{Grig--11}. 

\begin{prop}[\textbf{the counting measure is the average of the vertex spectral measures}]
For the adjacency operator $A$ of a finite graph $G = (V, E)$ with $n$ vertices,
$\frac{1}{n}$ times the sum of the vertex spectral measures at all vertices $v \in V$
is the counting measure on the spectrum of $A$.
\end{prop}

\section{Examples of finite graphs with and without dominant vertices}
\label{SectionFExamples}

\begin{exe}[\textbf{finite paths}]
\label{finitepaths}
Let $n \ge 2$.
The path $P_n$ has vertex set $\{1, 2, \hdots, n\}$
and edge set $\{ \{1, 2\}, \hdots, \{n-1, n\} \}$.
Denote by $A$ its adjacency operator, which is a symmetric $n$-by-$n$ matrix.
For $k \in \{1, 2, \hdots, n\}$, the vector $\xi_k$ defined by
$\xi_k (j) = \sin \left( \frac{j k}{n+1} \pi \right)$ for $j = 1, \hdots, n$
is an eigenvector of $A$
of eigenvalue $2 \cos \left( \frac{k}{n+1} \pi \right)$;
all these eigenvalues are simple, and there are no other eigenvalues.
\par

The vertex $j$ is dominant if and only if $j$ and $n+1$ are coprime.
Indeed, suppose first that $j$ and $n+1$ are coprime;
then $n+1$ does not divide $jk$, and therefore
$\langle \delta_j \sca \xi_k \rangle = \sin\left( \frac{jk}{n+1}\pi \right) \ne 0$
for $k \in \{1, \hdots, n\}$,
hence $j$ is dominant (see Proposition~\ref{domvertfg}~\ref{ivDEdomvertfg}).
Suppose then that $j$ and $n+1$ have a common divisor $\ell \ge 2$;
then, with $k = (n+1)/\ell$, we have
$\langle \delta_j \sca \xi_k \rangle = \sin\left( \frac{jk}{n+1}\pi \right)
= \sin\left( \frac{j}{\ell}\pi \right) = 0$,
hence $h$ is not dominant.

\[\begin{tikzpicture}
\vertex[fill] (a) at (0,1) {};
\vertex[] (b) at (1,1) {};
\vertex[] (c) at (2,1) {};
\vertex[] (d) at (3,1) {};
\vertex[fill] (e) at (4,1) {}; \path
(a) edge (b)
(b) edge (c)
(c) edge (d)
(d) edge (e) ;
\end{tikzpicture}\]

\[\begin{tikzpicture}
\vertex[fill] (a) at (0,1) {};
\vertex[fill] (b) at (1,1) {};
\vertex[fill] (c) at (2,1) {};
\vertex[fill] (d) at (3,1) {};
\vertex[fill] (e) at (4,1) {};
\vertex[fill] (f) at (5,1) {}; \path
(a) edge (b)
(b) edge (c)
(c) edge (d)
(d) edge (e)
(e) edge (f) ; 
\end{tikzpicture}\]

\[\begin{tikzpicture}
\vertex[fill] (a) at (0,1) {};
\vertex[] (b) at (1,1) {};
\vertex[fill] (c) at (2,1) {};
\vertex[] (d) at (3,1) {};
\vertex[fill] (e) at (4,1) {};
\vertex[] (f) at (5,1) {}; 
\vertex[fill] (g) at (6,1) {};	\path
(a) edge (b)
(b) edge (c)
(c) edge (d)
(d) edge (e) 
(e) edge (f) 
(f) edge (g) ;
\end{tikzpicture}\]

\[\begin{tikzpicture}
\vertex[fill] (a) at (0,1) {};
\vertex[fill] (b) at (1,1) {};
\vertex[] (c) at (2,1) {};
\vertex[fill] (d) at (3,1) {};
\vertex[fill] (e) at (4,1) {};
\vertex[] (f) at (5,1) {}; 
\vertex[fill] (g) at (6,1) {};
\vertex[fill] (h) at (7,1) {}; \path
(a) edge (b)
(b) edge (c)
(c) edge (d)
(d) edge (e) 
(e) edge (f) 
(f) edge (g) 
(g) edge (h) ;
\end{tikzpicture}\]
\vskip.1cm
\begin{center}
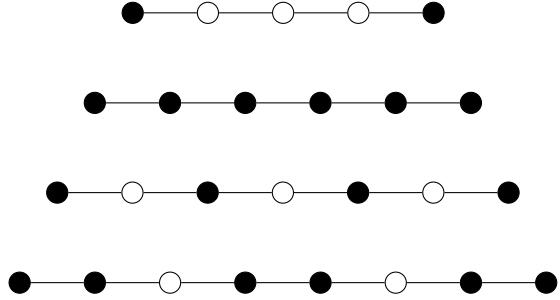
Figure 2. Graphs P$_5$, P$_6$, P$_7$, P$_8$ of Example~\ref{finitepaths}
with dominant vertices in black
\end{center}
\vskip.2cm

In particular, the two end vertices $1$ and $n$ are always dominant.
When $n+1$ is prime, all vertices are dominant.
If $n$ is odd and at least $3$, the middle vertex $\frac{n+1}{2}$ is not dominant.
\end{exe}

\begin{exe}[\textbf{complete bipartite graphs}]
\label{ComBipKmn}
Let $m$ and $n$ be two positive integers, such that $1 \le m \le n$.
The complete bipartite graph $K_{m,n}$
has vertex set $V = \{v_1, \hdots, v_m, w_1, \hdots, w_n\}$
and an edge $\{v_j, w_k\}$ for all $j = 1, \hdots.m$ and $k = 1, \hdots, n$.
Its adjacency operator has two simple eigenvalues $\sqrt{mn}$, $-\sqrt{mn}$,
and one eigenvalue $0$ of multiplicity $m+n-2$.
The function $\xi^+$ such that $\xi^+(v_j) = \sqrt{n}$ and $\xi^+(w_k) = \sqrt{m}$
for all $j = 1, \hdots, m$ and $k = 1, \hdots, n$
is an eigenvector of eigenvalue $\sqrt{mn}$,
the function $\xi^-$ such that $\xi^-(v_j) = \sqrt{n}$ and $\xi^-(w_k) = -\sqrt{m}$
for all $j = 1, \hdots, m$ and $k = 1, \hdots, n$
is an eigenvector of eigenvalue $-\sqrt{mn}$.
The space $\{ \xi \in \ell^2(V) : \sum_{j=1}^m \xi(v_j) = 0, \hskip.1cm \sum_{k=1}^n \xi(w_k) = 0 \}$
is the eigenspace of eigenvalue $0$, of dimension $m+n-2$.
\par

When $m=1$ and $n \ge 2$,
the star $K_{1, n}$ has a unique vertex $v_1$ which is not dominant,
and all leaves $w_1, \hdots, w_n$ are dominant.
When $2 \le m \le n$, all vertices of $K_{m,n}$ are dominant.
\end{exe}

\begin{exe}[\textbf{first graphs without dominant vertices}]
\label{CB1}
It is easy to find graphs without dominant vertices.
For example, let $G'$ and $G''$ be two finite graphs
such that at least one eigenvalue of $G'$ is not an eigenvalue of $G''$
and at least one eigenvalue of $G''$ is not an eigenvalue of $G'$;
the non-connected graph $G' \sqcup G''$
does not have any dominant vertex,
by Corollary~\ref{nullvectorsonly}.
Other disconnected graphs without any dominant vertex appear in
Remark~\ref{complement} and Example~\ref{regwithoutdominant}.
\par

We describe here a sequence of connected finite graphs without dominant vertices.
For $n \ge 2$, let $G_n$ be the graph with $n+2$ vertices $v_0, v'_0,$ $v_1, \hdots, v_n$
and $2n+1$ edges $\{v_0,v'_0\}, \{v_0, v_1\}, \{v'_0, v_1\}, \hdots, \{v_0, v_n\}, \{v'_0, v_n\}$.

\[\begin{tikzpicture}
\vertex[] (a) at (0,1) [label=left:$v_0$]{};
\vertex[] (b) at (0,3) [label=left:$v_0'$]{};
\vertex[] (x) at (1,2) [label=right:$v_1$]{};
\vertex[] (y) at (2,2) [label=right:$v_2$]{};	
\vertex[] (z) at (3,2) [label=right:$v_3$]{};	\path
(a) edge (b)
(a) edge (x)
(a) edge (y)
(a) edge (z)
(b) edge (x)
(b) edge (y)
(b) edge (z) ;
\end{tikzpicture}\]
\vskip.1cm
\begin{center}
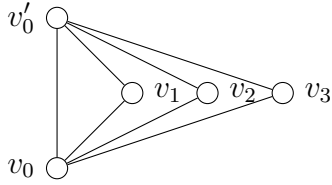
Figure 3. The graph $G_3$ of Example~\ref{CB1}, without dominant vertices
\end{center}
\vskip.2cm

The eigenvalues of $G_n$ are
$$
\begin{aligned}
\lambda^+_n &= \frac{1}{2}( 1 + \sqrt{ 8n+1 })
\hskip.5cm \text{simple Perron--Frobenius eigenvalue}
\\
\lambda^0_n &= 0
\hskip.5cm \text{of multiplicity $n-1$}
\\
\lambda^{-1}_n &= -1
\hskip.5cm \text{simple eigenvalue}
\\
\lambda^-_n &= \frac{1}{2}( 1 - \sqrt{ 8n+1 })
\hskip.5cm \text{simple eigenvalue}
\end{aligned}
$$
respectively with eigenvectors given by
$$
\begin{aligned}
\xi^+_n(v_0) = \xi^+_n(v'_0) = \frac{1}{4}( 1+ \sqrt{ 8n+1 }), \hskip.5cm
&\xi^+_n(v_1) = \cdots = \xi^+_n(v_n) = 1
\\
\xi^0_n(v_0) = \xi^0_n(v'_0) = 0, \hskip.5cm
&\xi^0_n(v_1) + \cdots + \xi^0_n(v_n) = 0
\\
\xi^{-1}_n(v_0) = 1, \hskip.1cm \xi^{-1}_n(v'_0) = -1, \hskip.5cm
&\xi^{-1}_n(v_1) = \cdots = \xi^{-1}_n(v_n) = 0
\\
\xi^-_n(v_0) = \xi^-_n(v'_0) = \frac{1}{4}( 1- \sqrt{ 8n+1 }), \hskip.5cm
&\xi^-_n(v_1) = \cdots = \xi^-_n(v_n) = 1 .
\end{aligned}
$$
Observe that $v_0$ and $v'_0$ are null vertices for the multiple eigenvalue $0$,
and that the other $v_1, v_2, \hdots, v_n$ are null vertices for the simple eigenvalue $-1$.
It follows from Corollary~\ref{nullvectorsonly}
that $G_n$ does not have any dominant vertex.
\par
Note that $G_n$ is an integral graph, i.e.,
that $\frac{1}{2}( 1 \pm \sqrt{ 8n+1 })$ are integers,
for $n= 3, 6, 10, 15, 21, 28, ...$.
\end{exe}

\begin{rem}
\label{complement}
There are thirty connected graphs on $2$, $3$, $4$ or $5$ vertices.
Five of them have no dominant vertex at all;
they are described by Figure~4.
Each of them is the complement of a non-connected graph
without dominant vertices.
(Recall that the complement of a graph $G = (V, E)$
is the graph $G' = (V, E')$ where two distinct vertices $v,w \in V$
are connected by an edge in $E'$ if an only if
they are not connected by an edge in $E$.)
The five graphs appear in \cite{CoSi--57}
and \cite[Appendix B, graphs $5, 11, 15, 17, 23$]{CvRS--97}.
The first and the third graph of Figure~4
are respectively graphs $G_2$ and $G_3$ of Example~\ref{CB1}.

\[\begin{tikzpicture}
\vertex[] (a) at (0,0) {};
\vertex[] (b) at (0,1.5) {};
\vertex[] (c) at (1.5,1.5) {};	
\vertex[] (d) at (1.5,0) {};
\vertex[] (e) at (3.3,0) {};
\vertex[] (f) at (4.7,0) {};
\vertex[] (g) at (5,1) {};
\vertex[] (h) at (4,1.5) {};
\vertex[] (j) at (3,1) {};
\vertex[] (m) at (6.3,0) {};
\vertex[] (n) at (7.9,0) {};
\vertex[] (p) at (7.1,0.7) {};
\vertex[] (q) at (7.1,1.2) {};
\vertex[] (r) at (7.1,1.7) {};
\path
(a) edge (b)
(b) edge (c)
(c) edge (d)
(d) edge (a)
(a) edge (c)
(f) edge (g)
(g) edge (h)
(h) edge (j)
(j) edge (e)
(e) edge (g)
(g) edge (j)
(j) edge (f)
(f) edge (h)
(h) edge (e)
(m) edge (n)
(m) edge (p)
(m) edge (q)
(m) edge (r)
(n) edge (p)
(n) edge (q)
(n) edge (r) ;
\end{tikzpicture}\]
\vskip.2cm
\[\begin{tikzpicture}
\vertex[] (a) at (0,0) {};
\vertex[] (b) at (0,1.4) {};
\vertex[] (c) at (0.7,0.7) {};	
\vertex[] (d) at (1.4,0) {};
\vertex[] (e) at (1.4,1.4) {};
\vertex[] (f) at (3.5,0) {};
\vertex[] (g) at (4.9,0) {};
\vertex[] (h) at (4.2,0.7) {};
\vertex[] (j) at (4.9,1.4) {};
\vertex[] (k) at (3.5,1.4) {};
\path
(a) edge (b)
(b) edge (e)
(e) edge (d)
(d) edge (a)
(c) edge (a)
(c) edge (b)
(c) edge (d)
(h) edge (f)
(h) edge (g)
(h) edge (j)
(h) edge (k)
(j) edge (k) ;
\end{tikzpicture}\]
\vskip.1cm
\begin{center}
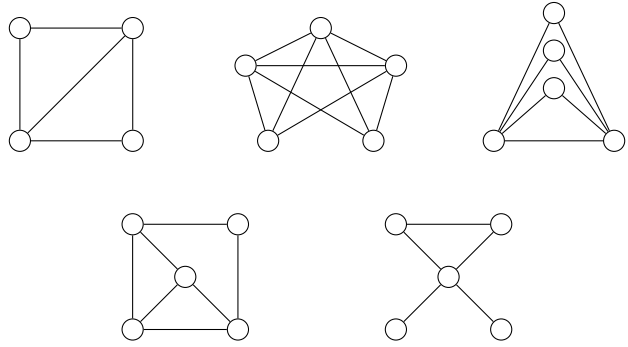
Figure 4. The connected graphs on 5 or less vertices without dominant vertices
\end{center}
\vskip.2cm
\end{rem}

\begin{exe}
\label{CB2}
For integers $n, k$ such that $n \ge 4$ and $2 \le k \le (n+1)/2$,
let $X_{n,k}$ denote the tree obtained from two paths $P_n$,
respectively with sets of vertices $v_1, \hdots, v_n$ and $v'_1, \hdots v'_n$
(in this order), by gluing $v_k$ and $v'_k$.
Thus $X_{n,k}$ has $2n-2$ edges and $2n-1$ vertices,
one of them of degree $4$, and $2n-6$ of degree $2$, and four leaves of degree $1$.
In the next proposition, we show that the $X_{n,k}$~'s with $k$ and $n-k+1$ coprime,
namely the graphs $X_{4,2}$, $X_{6,2}$, $X_{6,3}$, $X_{7,3}$, $X_{8,2}$, $X_{8,4}$, $\hdots$,
have no dominant vertices.
\end{exe}

\begin{prop}
[\textbf{a sequence of trees without dominant vertices}]
\label{CB2+}
Consider two integers $n, k$ such that $n \ge 4$ and $2 \le k \le (n+1)/2$,
and let $X_{n,k}$ be the tree defined in Example~\ref{CB2}.
Suppose moreover that $k$ and $n-k+1$ are coprime.
\par
Then the tree $X_{n,k}$ does not have any dominant vertex.
\end{prop}

\[\begin{tikzpicture}
\vertex[] (a) at (0,2) [label=above:$v_1$]{};
\vertex[] (b) at (1,2) [label=above:$v_2$]{};
\vertex[] (r) at (2,1) [label=above:$v_3$]{};
\vertex[] (c) at (3,2) [label=above:$v_4$]{};
\vertex[] (d) at (4,2) [label=above:$v_5$]{};
\vertex[] (e) at (5,2) [label=above:$v_6$]{};
\vertex[] (f) at (6,2) [label=above:$v_7$]{};
\vertex[] (u) at (0,0) [label=above:$v'_1$]{};
\vertex[] (v) at (1,0) [label=above:$v'_2$]{};
\vertex[] (w) at (3,0) [label=above:$v'_4$]{};
\vertex[] (x) at (4,0) [label=above:$v'_5$]{};
\vertex[] (y) at (5,0) [label=above:$v'_6$]{};
\vertex[] (z) at (6,0) [label=above:$v'_7$]{};	\path
(a) edge (b)
(b) edge (r)
(r) edge (c)
(c) edge (d)
(d) edge (e)
(e) edge (f)
(u) edge (v)
(v) edge (r)
(r) edge (w)
(w) edge (x)
(x) edge (y)
(y) edge (z) ;
\end{tikzpicture}\]
\vskip.1cm
\begin{center}
Figure 5. The graph $X_{7,3}$ of~\ref{CB2} and~\ref{CB2+}, without dominant vertices
\end{center}
\vskip.2cm

The following notion will be used in the proof of the proposition.
Let $G = (V, E)$ be a finite graph with $n = \vert V \vert$ vertices.
For a vector $\xi \in \ell^2(V)$, denote by~$\xi_w$ its $w^{{\rm th}}$ coordinate
with respect to the basis $(\delta_w)_{w \in V}$ of $\ell^2(V)$.
Define the \textbf{walk matrix} $W^{(v)}$ of $G$ at $v$ to be the matrix
with columns $\delta_v, A \delta_v, A^2 \delta_v, \hdots, A^{n-1} \delta_v$;
it is a square matrix with $n$ rows (indexed by $V$)
and $n$ columns (indexed by $\{ 0, 1, \hdots, n-1 \}$),
with entries in $\N$.
For indices $w \in V$ and $\ell \in \{ 0, 1, \hdots, n-1 \}$,
the $(w,\ell)$-entry $(A^\ell \delta_v)_w$ of $W^{(v)}$
is the number of walks of length $\ell$ in $G$ which start at $v$ and end at $w$
(see Proposition~\ref{localspdominantvertex}).
\par

The first $s$ column of $W^{(v)}$ appear in Condition~\ref{iiDEdomvertfg}
of Proposition~\ref{domvertfg}.
The matrix $W^{(v)}$ appears in the literature,
for example as a particular case of the walk matrices at subsets of $V$,
considered in \cite{LiSi--22}.
The matrix $\sum_{v \in V} W^{(v)}$ appears in various places,
including \cite[Section 14.2]{BrHa--12}, with the same name of walk matrix.

\begin{proof}[Proof of Proposition~\ref{CB2+}]
Denote by $A$ the adjacency operator of $X_{n,k}$.
\par
Claim:
\emph{all eigenvalues of $A$ are simple.}
\newline
To show this, we assume that $A$ has an eigenvalue $\lambda$
of multiplicity at least $2$, and we will obtain a contradiction.
\par

The following has been established by Godsil;
see \cite[Lemma~2]{Gods--84} or \cite[Theorem 5.1.2]{BrHa--12}.
\par

\emph{Let $X$ be a tree, $A_X$ its adjacency operator,
and $\lambda$ an eigenvalue of $A_X$ of multiplicity $m \ge 2$.
Let $P$ be a subgraph of $X$ which is a path,
and $X \smallsetminus P$ the graph obtained from $X$
by deleting the vertices of $P$ and the edges incident to them.
Then $\lambda$ is an eigenvalue of $X \smallsetminus P$
of multiplicity at least $m-1$.}
\par

We apply this to $X = X_{n,k}$.
Consider first the path $P'$ from $v_n$ to $v'_n$ (the notation is that of Example~\ref{CB2});
since $X_{n,k} \smallsetminus P'$ is the disjoint union of
two copies of the path $P_{k-1}$ with $k-1$ vertices,
$\lambda$ is an eigenvalue of $P_{k-1}$.
Consider then the path $P''$ from $v_1$ to $v'_1$;
since $X_{n,k} \smallsetminus P''$ is the disjoint union of
two copies of the path $P_{n-k}$,
$\lambda$ is an eigenvalue of $P_{n-k}$.
But the spectra of $P_{k-1}$ and $P_{n-k}$ are disjoint,
because $k$ and $n-k+1$ are coprime, and this is the announced contradiction.
\par

It follows that $A$ has simple eigenvalues only,
in other words that $A$ has $2n-1$ distinct eigenvalues.
\par

It remains to show that any vertex $v$ is not dominant.
By Proposition~\ref{domvertfg}, it suffices to show that the vectors
$\delta_v, A \delta_v, A^2 \delta_v, \hdots, A^{2n-2} \delta_v$ are linearly dependent,
i.e., that
\emph{the rank of the walk matrix $W^{(v)}$ at $v$ is strictly less than $2n-1$.}
There are two cases to consider.
\par

Suppose first that $v$ is either $v_i$ or $v'_i$ for some $i \in \{1, \hdots, k\}$.
For any $j \in \{k+1, \hdots, n\}$ and for any $\ell \in \{0, 1, \hdots, n-2\}$,
the number of walks in $X_{n,k}$ of length $\ell$
starting at $v$ and ending at $v_j$ is equal to
the number of such walks starting at $v$ and ending at $v'_j$.
Thus $\left( W^{(v)} \right)_{v_j, \ell} = \left( W^{(v)} \right)_{v'_j, \ell}$ \hskip.1cm ;
in other words, row $v_j$ and row $v'_j$ of $W^{(v)}$ coincide.
%
%
Thus the matrix $W^{(v)}$ has $n-k$ pairs of equal rows,
and consequently its rank is at most $(2n-1) - (n-k) = n+k-1$,
in particular its rank is strictly less than $2n-1$.
\par

Suppose now that $v$ is either $v_i$ or $v'_i$ for some $i \in \{k, \hdots, n\}$.
Similarly, the matrix $W^{(v)}$ has $k-1$ pairs of equal rows,
its rank is at most $(2n-1) - (k-1) = 2n-k$,
in particular its rank is strictly less than $2n-1$.
\end{proof}

\begin{exe}[\textbf{regular connected graphs without dominant vertices}]
\label{regwithoutdominant}
For $m \ge 3$, denote by $C_m$ the cycle with $m$ vertices.
Denote by $v_0, v_1, \hdots, v_{m-1}$ the vertices of $C_m$;
the edges are $\{v_{k-1}, v_k \}$ for $k \in \{ 0, 1, \hdots, m-1\}$
(with $-1 = m-1$).
Set $\zeta = \exp \left( i \frac{2}{m} \pi \right)$.
For $j = 0, 1, \hdots, m-1$,
define the vector $\xi_j \in \ell^2(C_m)$ by $\xi_j(v_k) = \zeta^{k j}$
for $k = 0, 1, \hdots, m-1$.
Then $\xi_j$ is an eigenvector of the adjacency operator of $C_m$
of eigenvalue $2 \cos \left( \frac{2j}{m} \pi \right)$.
This shows that the spectrum of $C_m$ consists of
$2$ (a simple eigenvalue), 
$2\cos \left( \frac{2j}{m} \pi \right)$ for $1 \le j \le \frac{m-1}{2}$ (eigenvalues of multiplicity~$2$),
and $-2$ when $m$ is even (a simple eigenvalue).
All vertices are dominant by Proposition~\ref{transimpliesdom}.
\par

Let $m, n$ be two coprime integers, both $\ge 3$.
The spectra of $C_m$ and $C_n$ contain $2$, and are otherwise disjoint.
The non-connected graph $C_m \sqcup C_n$ is $2$-regular.
All eigenvalues have multiplicity $2$, except $-2$ which is a simple eigenvalue
if one of $m$, $n$ is even.
Denote by $C'_{m,n}$ the complement of the graph $C_m \sqcup C_n$.
\par

The following is well-known:
\emph{
Let $\xi_* = \left( \xi_{{\rm PF}}, \xi_2, \hdots, \xi_{m+n} \right)$
be an orthonormal basis of $\ell^2(V)$ of eigenvectors of $C_m \sqcup C_n$,
where $\xi_{{\rm PF}} (v) = \frac{ 1 }{ \sqrt{m+n} }$ for any vertex $v$ of $C_m \sqcup C_n$.
Let $2, \lambda_2, \hdots, \lambda_{m+n}$ be the corresponding eigenvalues.
Then $\xi_*$ is also a basis of eigenvectors of the complement graph $C'_{m,n}$
and the corresponding eigenvalues are
$m+n-3, -1-\lambda_2, \hdots, -1-\lambda_{m+n}$.
}
See for example~\cite[Lemma 8.5.1]{GoRo--01}.
In particular, any eigenvector of $C'_{m,n}$ for an eigenvalue $\ne m+n-3$
has either all vertices of $C_m$ or all vertices of $C_n$ as null vertices.
It follows that $C'_{m,n}$ does not have any dominant vertex.
\par

The smallest of these graphs is $C'_{3,4}$.
It has $7$ vertices, it is $4$-regular, its spectrum is
$$
4, \hskip.2cm 1, \hskip.2cm 0, \hskip.2cm 0, \hskip.2cm -1, \hskip.2cm -1, \hskip.2cm -3
$$
(because the eigenvalues of $C_3 \sqcup C_4$ are $2, -1, -1, 2, 0, 0, -2$)
and it does not have any dominant vertex.

\[\begin{tikzpicture}
\vertex[] (a) at (8,0) {};
\vertex[] (b) at (11,-0.3) {};
\vertex[] (c) at (14,0) {};
\vertex[] (d) at (14,2) {};
\vertex[] (e) at (11,2.3) {};
\vertex[] (f) at (8,2) {};
\vertex[] (g) at (11,1) {}; \path
(a) edge (b)
(b) edge (c)
(c) edge (d)
(d) edge (e)
(e) edge (f)
(f) edge (a)
(a) edge (e)
(e) edge (c)
(c) edge (g)
(g) edge (a)
(f) edge (b)
(b) edge (d)
(d) edge (g)
(g) edge (f) ;
\end{tikzpicture}\]
\vskip.1cm
\begin{center}
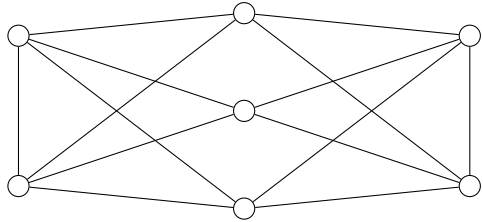
Figure 6. The graph $C'_{3,4}$ of Example~\ref{regwithoutdominant}
\\
without dominant vertices
\end{center}
\vskip.2cm
\end{exe}

Recall that a connected finite graph of diameter $d$ has at least $d+1$ eigenvalues
\cite[Corollary~2.7]{Bigg--93}.

\begin{prop}
\label{s=d+1}
Let $G$ be a connected finite graph of diameter $d \ge 1$.
Assume that $G$ has exactly $d+1$ eigenvalues.
Let $v_0, v_d \in V$ be a pair of vertices at distance $d$ from one another.
\par
Then $v_0$ and $v_d$ are dominant vertices in $G$.
\end{prop}

\begin{proof}
Choose vertices $v_1, \hdots, v_{d-1}$ such that
$v_0, v_1, \hdots, v_{d-1}, v_d$ are the vertices, in this order,
of a path connecting $v_0$ to $v_d$.
Write $\delta_j$ for $\delta_{v_j}$.
For $k \in \{0, 1, \hdots, d\}$,
$$
\langle A^k \delta_0 \sca \delta_j \rangle
\left\{\begin{aligned}
\ne 0 \hskip.2cm &\text{if} \hskip.2cm j = k
\\
= 0 \hskip.2cm &\text{if} \hskip.2cm j > k
\end{aligned} 
\right.
\hskip.5cm \text{so that} \hskip.5cm
( A^k \delta_0 )(v_j) 
\left\{\begin{aligned}
\ne 0 \hskip.2cm &\text{if} \hskip.2cm j = k
\\
= 0 \hskip.2cm &\text{if} \hskip.2cm j > k
\end{aligned} 
\right.
$$
(when $k = d$, the two lines ``$= 0$ if $j > k$'' should be ignored)
and therefore $\delta_0, A \delta_0, \hdots, A^d \delta_0$
are linearly independent.
By Proposition~\ref{domvertfg}, the vertex $v_0$ is dominant.
Similarly, $v_d$ is dominant.
\end{proof}

Proposition~\ref{s=d+1} implies for example that the end vertices are dominant
in a finite path
(as already observed in Examples~\ref{finitepaths} above),
or that all vertices are dominant in a complete graph.

\begin{exe}[\textbf{all vertices are dominant in distance-regular graphs}]
\label{DS}
A distance-regular graph is a connected finite graph $G$ of some diameter $d$ such that,
for two vertices $v, w$ of $G$ and two integers $j, h \in \{0, 1, \hdots, d\}$,
the number of vertices $u$ of $G$ such that $\vert u-v \vert = j$ and $\vert u-w \vert = h$
depends only on $\vert w - v \vert$.
For such a graph, the number of eigenvalues of the adjacency operator
is exactly $d+1$ \cite[Theorem 20.7]{Bigg--93}
and, for any vertex $v$, there exists a vertex $w$ at distance exactly $d$ from $v$.
Distance-regular graphs of diameter $2$ are called strongly regular graphs,
\par

All vertices are dominant in a finite distance-regular graph.
This follows from Proposition~\ref{s=d+1},
or alternatively from Proposition~\ref{wrg}
since distance-regular graphs are walk-regular.
\end{exe}

\begin{rem}
\label{remdr}
A graph is distance-transitive if its automorphism group
is transitive on equidistant pairs of vertices.
These graphs are clearly vertex-transitive
hence all their vertices are dominant by Proposition~\ref{transimpliesdom}.
Alternatively, these graphs are distance-regular
hence all their vertices are dominant by the previous proposition.
However, as it is well-known, 
a distance-regular graph need not be distance-transitive;
see for example \cite[Additional result 20c]{Bigg--93}.
\par

A connected finite graph with two eigenvalues is a complete graph.
A connected finite regular graph with three eigenvalues
is a strongly regular graph.
It is known that a connected finite regular graph with four eigenvalues
is walk-regular~\cite[Corollary~15.1.2]{BrHa--12},
and therefore all its vertices are dominant by Proposition~\ref{wrg}.
A connected finite regular graph with four eigenvalues
need not be distance regular \cite{vDam--95}.
\end{rem}

\begin{exe}
\label{CB3+}
For integers $k \ge 2, r \ge 1$. 
Let $T_{k,r}$ denote the $k$-ary rooted tree of depth $r$.
The tree $T_{k,1}$ is the star with one root of degree $k$ and $k$ leaves,
for $r \ge 2$, the tree $T_{k,r}$ is obtained by attaching $k$ new leaves to each leaf of $T_{k, r-1}$.
Thus $T_{k,r}$ has first one root of degree $k$, 
then $k + k^2 + \cdots + k^{r-1}$ vertices of degree $k+1$,
and finally $k^r$ leaves at distance $r$ from the root.
Altogether $T_{k,r}$ has $1 + k + \cdots + k^r = (k^{r+1} - 1)/(k-1)$ vertices. 
Eigenvalues and eigenvectors of these trees have been computed in \cite{DeRo--20}.
In particular, the results there imply:
\begin{center}
\emph{In the $k$-ary rooted tree $T_{k,r}$, with $k \ge 2$ and $r \ge 1$,
\\
a vertex is dominant if and only if it is a leaf.}
\end{center}
The particular case $T_{2,2}$ appears again ad $\widetilde{D_6}$
in Example~\ref{exCoxAffine}.
\end{exe}

\begin{exe}[\textbf{Coxeter graphs, types $A$, $D$, and $E$}]
\label{exCoxFinite}
For a Coxeter graph of type $E_8$,
the characteristic polynomial is irreducible
and all vertices are dominant;
see Example~\ref{exE8}.
The other connected Coxeter graphs of finite type, those of the types
$A_\ell$ ($\ell \ge 2$),
$D_\ell$ ($\ell \ge 4$), $E_6$, and~$E_7$,
have characteristic polynomials which are reducible over $\Q$,
as we check now.

\vskip.2cm

The Coxeter graphs of type $A$ are the finite paths.
Let $P_\ell$ denote here the characteristic polynomial of the Coxeter graph of type~$A_\ell$.
We have $P_1(T) = T$, $P_2(T) = T^2-1$, and $P_{\ell+1}(T) = TP_\ell(T) - P_{\ell-1}(T)$
for all $\ell \ge 3$; it follows that $P_\ell(T) = U_\ell(T/2)$,
where $U_\ell$ is the Chebyshev polynomial of the second kind of degree $\ell$.
Since $U_\ell$ is well-known to be reducible for all $\ell \ge 2$
(see for example \cite{CaCe--17}), the polynomial $P_\ell$ is also reducible for all $\ell \ge 2$.
\par

Dominant vertices of $P_\ell$ have been described in Example~\ref{finitepaths}.

\[\begin{tikzpicture}
\vertex[] (a) at (0,1) [label=above:$v_1$]{};
\vertex[] (b) at (1,1) [label=above:$v_2$]{};
\vertex[] (c) at (2,1) [label=above:$v_3$]{};
\vertex[] (d) at (3,1) [label=above:$v_4$]{};
\vertex[] (e) at (4,1) [label=above:$v_5$]{};
\vertex[] (f) at (5,1) [label=above:$v_6$]{};
\vertex[] (g) at (6,1) [label=above:$v_7$]{};
\vertex[fill] (x) at (7,2) [label=above:$v_8$]{};
\vertex[fill] (y) at (7,0) [label=above:$v_9$]{}; \path
(a) edge (b)
(b) edge (c)
(c) edge (d)
(d) edge (e)
(e) edge (f)
(f) edge (g)
(g) edge (x)
(g) edge (y) ;
\end{tikzpicture}\]

\[\begin{tikzpicture}
\vertex[fill] (a) at (0,1) {};
\vertex[] (b) at (1,1) {};
\vertex[fill] (c) at (2,1) {};
\vertex[] (d) at (3,1) {};
\vertex[fill] (e) at (4,1) {};
\vertex[] (f) at (5,1) {};
\vertex[fill] (g) at (6,1) {};
\vertex[] (h) at (7,1) {};
\vertex[fill] (x) at (8,2) {};
\vertex[fill] (y) at (8,0) {}; \path
(a) edge (b)
(b) edge (c)
(c) edge (d)
(d) edge (e)
(e) edge (f)
(f) edge (g)
(g) edge (h)
(h) edge (x)
(h) edge (y) ;
\end{tikzpicture}\]

\vskip.1cm
\begin{center}
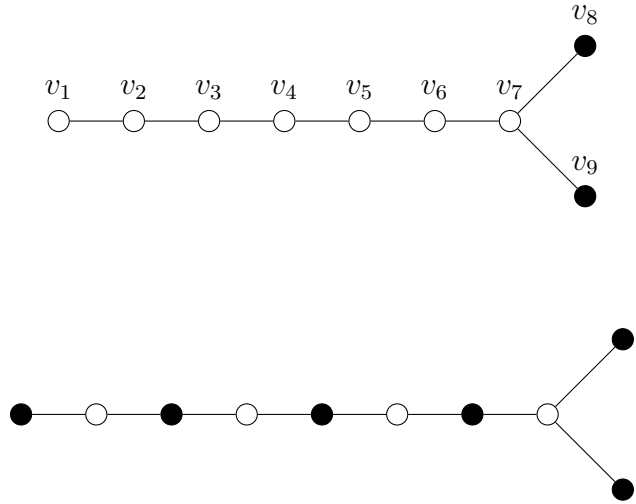
Figure 7. The graph 
$D_9$ and $D_{10}$ of Example~\ref{exCoxFinite}
\\
with dominant vertices in black
\end{center} 
\vskip.2cm

For $\ell \ge 3$, the Coxeter graph of type $D_\ell$
has vertex set $V = \{v_1, v_2, \hdots, v_\ell\}$ and edge set
$\big\{ \{ v_1, v_2 \}, \hdots, \{ v_{\ell-3}, v_{\ell-2} \},
\{ v_{\ell-2}, v_{\ell-1} \}, \{ v_{\ell-2}, v_{\ell} \} \big\}$;
the graph $D_2$ has two isolated vertices, and no edge.
Let $Q_\ell$ be the characteristic polynomial of the Coxeter graph of type $D_\ell$.
We have $Q_2(T) = T^2$, $Q_3(T) = T^3 - 2T$,
and $Q_{\ell+1}(T) = TQ_\ell(T) - Q_{\ell-1}(T)$ for all $\ell \ge 3$.
It follows that $Q_\ell$ is divisible by $T^2$ for all even $\ell \ge 2$
and divisible by $T$ for all odd $\ell \ge 3$,
in particular that $Q_\ell$ is reducible for all $\ell \ge 2$.
From now on, we assume that $\ell \ge 4$.
\par

The vector $\xi_0 \in \ell^2(V)$ defined by
$$
\xi_0 (v_1) = \cdots = \xi_0(v_{\ell - 2}) = 0 ,
\hskip.5cm \xi_0(v_{\ell - 1}) = 1 ,
\hskip.2cm
\xi_0(v_\ell) = -1 ,
$$
is an eigenvector of eigenvalue $0$.
For $k \in \{1, 2, \hdots, \ell-1\}$,
the vector $\xi_k$ defined by
$$
\begin{aligned}
& \xi_k (v_j) = 2 \cos \left( \frac{ (2k-1) (\ell - j - 1) }{ 2\ell - 2 } \pi \right)
\hskip.5cm \text{for} \hskip.2cm
j \in \{1, 2, \hdots, \ell - 2\} ,
\\
& \xi_k (v_{\ell - 1}) = \xi_k (v_\ell) = 1 ,
\end{aligned}
$$
is an eigenvector of eigenvalue $2 \cos \left( \frac{ 2k-1 }{ 2\ell - 2 } \pi \right)$.
\par

If $\ell = 4$, the graph $D_4$ is the complete bipartite graph $K_{1,3}$:
it has three dominant vertices which are its three leaves,
Assume now that $\ell \ge 5$.
For $k \in \{0, 1, \hdots, \ell-1\}$,
the values $\xi_k(v_{\ell-1})$ and $\xi_k(v_\ell)$ are not zero.
By Condition~\ref{ivDEproplinalgsa} of Proposition~\ref{proplinalgsa},
the vertices $v_{\ell-1}$ and $v_\ell$ are dominant.
\par

Consider the case of $\ell$ odd and at least $5$.
All eigenvalues are simple.
For $j \in \{1, \hdots, \ell-2\}$,
the orthogonal projection of $\delta_{v_j}$ onto the eigenspace of eigenvalue~$0$
is zero, and it follows that $v_j$ is not dominant.
\par

Consider now the case of $\ell$ even and at least $6$.
The eigenvalue $0$ has multiplicity~$2$,
it has eigenvectors $\xi_0$ and $\xi_{\ell/2}$,
and the other eigenvalues are simple.
For $j$ odd in $\{1, 3, \hdots, \ell-3\}$ and $k \in \{1, \hdots, \ell-1\}$,
the scalar product
$\langle \delta_{v_j} \sca \xi_k \rangle = 2 \cos \left( \frac{ (2k-1) (\ell - j - 1) }{ 2\ell - 2 } \pi \right)$
is not $0$; indeed, if it were $0$, we would have
$\frac{ (2k-1) (\ell - j - 1) }{ 2\ell - 2 } \pi = \frac{\pi}{2} + n\pi$ for some $n \in \Z$,
namely $(2k-1) (\ell - j - 1) = (\ell - 1)(2n+1)$,
but this is impossible because $\ell - j - 1$ is even and $(\ell - 1)(2n + 1)$ is odd.
It follows that $v_1, v_3, \hdots, v_{\ell-3}$ are dominant.
For $j$ even in $\{2, 4, \hdots, \ell-2 \}$ and $k = \ell/2$,
the scalar products
$\langle \delta_{v_j} \sca \xi_{\ell/2} \rangle = 2 \cos \left( \frac{\ell - j - 1}{ 2 } \pi \right)$
and $\langle \delta_{v_j} \sca \xi_0 \rangle$ are both $0$;
it follows that the orthogonal projection of $\delta_{v_j}$
onto the eigenspace of eigenvalue $0$ is zero, and therefore that $v_j$ is not dominant.

\[\begin{tikzpicture}
\vertex[fill] (a) at (0,1) {};
\vertex[fill] (b) at (1,1) {};
\vertex[] (r) at (2,1) {};
\vertex[] (h) at (2,2) {};
\vertex[fill] (x) at (3,1) {};
\vertex[fill] (y) at (4,1) {}; \path
(a) edge (b)
(b) edge (r)
(r) edge (x)
(r) edge (h)
(x) edge (y) ;
\end{tikzpicture}\]

\[\begin{tikzpicture}
\vertex[] (a) at (0,1) {};
\vertex[] (b) at (1,1) {};
\vertex[] (r) at (2,1) {};
\vertex[fill] (h) at (2,2) {};
\vertex[fill] (x) at (3,1) {};
\vertex[] (y) at (4,1) {};
\vertex[fill] (z) at (5,1) {}; \path
(a) edge (b)
(b) edge (r)
(r) edge (x)
(r) edge (h)
(x) edge (y)
(y) edge (z) ;
\end{tikzpicture}\]
\vskip.1cm
\begin{center}
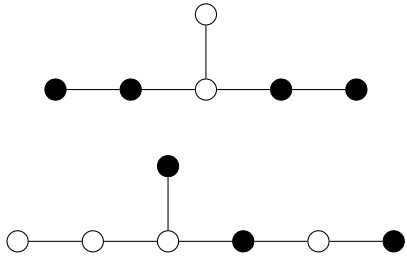
Figure 8. Graph $E_6$ and $E_7$ of Example~\ref{exCoxFinite}
\\
with dominant vertices in black
\end{center}
\vskip.2cm

Direct computations show that the characteristic polynomials of $E_6$ and $E_7$
$$
\begin{aligned}
P_{E6}(T) &= T^6 - 5T^4 + 5T^2 - 1 = (T-1)(T+1)(T^4- 4T^2 + 1) ,
\\
P_{E7}(T) &= T^7 - 6T^5 - 9T^3 + 3T = T(T^6 - 6T^4 -9T^2 + 3) ,
\end{aligned}
$$
are reducible over $\Q$.
Explicit computations again show that
all eigenvalues are simple,
that the graph $E_6$ has four dominant vertices
(not the vertex of degree $3$, and not the leave adjacent to it),
and that the graph $E_7$ has three dominant vertices
(the leaf adjacent to the vertex of degree $3$,
the leaf at distance $3$ of this vertex of degree $3$, 
and the vertex at distance $2$ of this last leaf).
See Figure 8.
\par

We have already noted in Example~\ref{exE8}
that all vertices are dominant in the graph $E_8$.
\end{exe}

\begin{exe}
[\textbf{Coxeter graphs, types $\widetilde A$, $\widetilde D$, and $\widetilde E$}]
\label{exCoxAffine}
These are more precisely the graphs of type
$\widetilde{A_\ell}$ $(\ell \ge 2)$, 
$\widetilde{D_\ell}$ $(\ell \ge 4)$,
and $\widetilde{E_\ell}$ $(\ell = 6, 7, 8)$;
the graph with subscript $\ell$ has $\ell+1$ vertices.
These graphs are the finite connected graphs with
maximal eigenvalue exactly $2$ \cite[Section~3.1.1]{BrHa--12}.
It follows that the characteristic polynomials of these graphs
are reducible, divisible by $T-2$.
\par

For $\ell \ge 2$, the Coxeter graph of type $\widetilde{A_\ell}$ is the cycle $C_{\ell+1}$
with $\ell+1$ vertices.
In Example~\ref{regwithoutdominant},
we have already written the list of its eigenvalues.
We have also observed that all vertices are dominant (Proposition~\ref{transimpliesdom}).
\par

The Coxeter graph of type $\widetilde{D_4}$ is the complete bipartite graph $K_{1,4}$
of Example~\ref{ComBipKmn};
it has four dominant vertices, which are its four leaves.
For $\ell \ge 5$, 
the graph of type $\widetilde{D_\ell}$
has vertex set $V_\ell = \{v_0, v_1, \hdots, v_\ell\}$
and edge set the edge set of $D_\ell$ plus the edge $\{v_0, v_2\}$.
The vectors $\xi'_0$ and $\xi''_0$ defined by
$$
\begin{aligned}
\xi'_0(v_0) = 1, \hskip.5cm \xi'_0(v_1) = -1, \hskip.5cm
& \xi'_0(v_j) = 0 \hskip.5cm \text{for} \hskip.2cm 2 \le j \le \ell ,
\\
\xi''_0(v_j) = 0 \hskip.5cm \text{for} \hskip.2cm 0 \le j \le \ell - 2 \hskip.5cm
& \xi''_0(v_{\ell-1}) = 1, \hskip.5cm \xi''_0(v_\ell) = -1 ,
\end{aligned}
$$
are eigenvectors of eigenvalue $0$.
For $k \in \{0, 1, \hdots, \ell-1 \}$, the vector $\xi_k$ defined by
$$
\begin{aligned}
& \xi_k(v_0) = \xi_k(v_1) = 1, \hskip.5cm \xi_k(v_{\ell-1}) = \xi_k(v_\ell) = (-1)^k ,
\\
& \xi_k(v_j) = 2 \cos \Big( \frac{ (j-1) k }{\ell-2 } \pi \Big)
\hskip.5cm \text{for} \hskip.2cm 2 \le j \le \ell-2 ,
\end{aligned}
$$
is an eigenvector of eigenvalue $\lambda_k = 2 \cos \frac{ k \pi }{\ell-2 }$.
Note that $A \xi_0 = 2 \xi_0$, so that $\xi_0$ is the Perron--Frobenius eigenvector
(the value of $\xi_0$ is $1$ on leaves and is two on other vertices,
in particular it is positive on all vertices);
moreover $\xi_{\ell-1}(v_j) = \pm \xi_0(v_j)$ for all $j$
with signs such that $A \xi_{\ell-1} = -2\xi_{\ell-1}$.
All eigenvectors above take non-zero values on leaves,
so that the four leaves of $\widetilde{D_\ell}$ are always dominant.
To decide whether there are other dominant vertices,
we distinguishes two cases, according to the parity of $\ell$.
\par

If $\ell$ is odd (hence $\vert V_\ell \vert$ is even), $\lambda_k \ne 0$ for $k \in \{0, 1, \hdots, \ell-1\}$,
hence the multiplicity of the eigenvalue $0$ is two.
All eigenvectors of eigenvalue $0$ (such as $\xi'_0$ and $\xi''_0$)
vanish on $v_j$ for $2 \le j \le \ell-2$,
and it follows that the latter vertices are not dominant.
Thus $\widetilde{D_\ell}$ has exactly four dominant vectors, which are its four leaves.
\par

If $\ell$ is even, the multiplicity of the eigenvalue $0$ is three,
and the corresponding eigenspace is spanned by $\xi'_0, \xi''_0$, and $\xi_{(\ell-2)/2}$.
For some values of $\ell$, the graph $\widetilde{D_\ell}$ has more dominant vertices than its leaves,
and for others it has not.
For example, the graph $\widetilde{D_6}$ has four dominant vertices,
and the graph $\widetilde{D_8}$ has six,
as shown in black on Figure 9.

\[\begin{tikzpicture}
\vertex[fill] (u) at (-1,2) [label=above:$v_0$]{};
\vertex[fill] (v) at (-1,0) [label=above:$v_1$]{};
\vertex[] (a) at (0,1) [label=above:$v_2$]{};
\vertex[] (b) at (1,1) [label=above:$v_3$]{};
\vertex[] (g) at (2,1) [label=above:$v_4$]{};
\vertex[fill] (x) at (3,2) [label=above:$v_5$]{};
\vertex[fill] (y) at (3,0) [label=above:$v_6$]{}; \path
(u) edge (a)
(v) edge (a)
(a) edge (b)
(b) edge (g)
(g) edge (x)
(g) edge (y) ;
\end{tikzpicture}\]

\[\begin{tikzpicture}
\vertex[fill] (u) at (-1,2) {};
\vertex[fill] (v) at (-1,0) {};
\vertex[] (a) at (0,1) {};
\vertex[fill] (b) at (1,1) {};
\vertex[] (c) at (2,1) {};
\vertex[fill] (d) at (3,1) {};
\vertex[] (e) at (4,1) {};
\vertex[fill] (x) at (5,2) {};
\vertex[fill] (y) at (5,0) {}; \path
(u) edge (a)
(v) edge (a)
(a) edge (b)
(b) edge (c)
(c) edge (d)
(d) edge (e)
(e) edge (x)
(e) edge (y) ;
\end{tikzpicture}\]

\vskip.1cm
\begin{center}
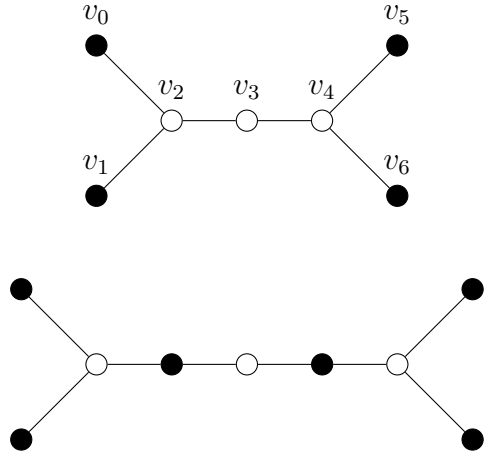
Figure 9. The graph 
$\widetilde{D_{6}}$ and $\widetilde{D_{8}}$ of Example~\ref{exCoxAffine}
\\
with dominant vertices in black
\end{center} 
\vskip.2cm

The graphs $\widetilde{E_6}$, $\widetilde{E_7}$, and $\widetilde{E_8}$,
have respectively three, two, and one,
dominant vertices.
The necessary computations are tedious but straightforward,
and we do not give the details,
Dominant vertices are shown in Figure 10.

\[\begin{tikzpicture}
\vertex[fill] (b) at (1,1) {};
\vertex[] (c) at (2,1) {};
\vertex[] (r) at (3,1) {};
\vertex[] (h) at (3,2) {};
\vertex[fill] (k) at (3,3) {};
\vertex[] (x) at (4,1) {};
\vertex[fill] (y) at (5,1) {}; \path
(b) edge (c)
(c) edge (r)
(r) edge (x)
(r) edge (h)
(h) edge (k)
(x) edge (y) ;
\end{tikzpicture}\]

\[\begin{tikzpicture}
\vertex[fill] (a) at (0,1) {};
\vertex[] (b) at (1,1) {};
\vertex[] (c) at (2,1) {};
\vertex[] (r) at (3,1) {};
\vertex[] (h) at (3,2) {};
\vertex[] (x) at (4,1) {};
\vertex[] (y) at (5,1) {};
\vertex[fill] (z) at (6,1) {}; \path
(a) edge (b)
(b) edge (c)
(c) edge (r)
(r) edge (x)
(r) edge (h)
(x) edge (y)
(y) edge (z) ;
\end{tikzpicture}\]

\[\begin{tikzpicture}
\vertex[] (a) at (0,1) {};
\vertex[] (b) at (1,1) {};
\vertex[] (r) at (2,1) {};
\vertex[] (h) at (2,2) {};
\vertex[] (x) at (3,1) {};
\vertex[] (y) at (4,1) {};
\vertex[] (z) at (5,1) {};
\vertex[] (t) at (6,1) {};
\vertex[fill] (u) at (7,1)[label=above:$v$] {};	\path
(a) edge (b)
(b) edge (r)
(r) edge (x)
(r) edge (h)
(x) edge (y)
(y) edge (z)
(z) edge (t)
(t) edge (u) ;
\end{tikzpicture}\]

\vskip.1cm
\begin{center}
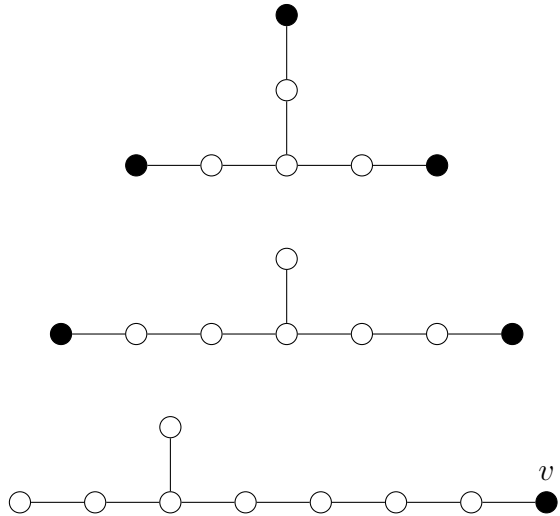
Figure 10. The graph 
$\widetilde{E_6}$, $\widetilde{E_7}$, and $\widetilde{E_8}$ of Example~\ref{exCoxAffine}
\\
with dominant vertices in black
\end{center} 
\vskip.2cm

The graph $\widetilde{E_8}$
has a remarkable property:
it has a unique vertex which is dominant;
this vertex is the leaf $v$ at distance $5$ from the vertex of degree~$3$.
\end{exe} 

\section{On a perturbation of the free Jacobi matrix}
\label{S+EofJc}

The object of this section is to determine the spectrum
and the eigenvalues of the operator $J_a$ defined below,
and to show that all basis vectors $\delta_n$ are cyclic for~$J_a$.
This will be used for the proof of Theorem~\ref{exStars}.
The results are known, see in particular \cite{LeNy--15} and \cite{Yafa--17}.

\begin{defn}
\label{defJa}
Consider the Hilbert space $\ell^2(\N)$,
its standard orthonormal basis $\delta_* = (\delta_n)_{n \ge 0}$,
a positive number $a > 0$, and the $\N$-by-$\N$ matrix
$$
J_a = \begin{pmatrix}
0 & a & 0 & 0 & 0 & \cdots
\\
a & 0 & 1 & 0 & 0 & \cdots
\\
0 & 1 & 0 & 1 & 0 & \cdots
\\
0 & 0 & 1 & 0 & 1 & \cdots
\\
0 & 0 & 0 & 1 & 0 & \cdots
\\
\vdots & \vdots & \vdots & \vdots & \vdots & \ddots
\end{pmatrix} .
\leqno{(Jac_a})
$$
This is the matrix with respect to the basis $\delta_*$
of a bounded self-adjoint operator on $\ell^2(\N)$
that we denote again by $J_a$ and which is called a \textbf{Jacobi operator}.
We may view $J_a$ as a perturbation of the matrix $J_1$,
often called the \textbf{free Jacobi matrix},
for example in \cite[Chapter 1]{Simo--11}.
\end{defn}

Let $S$ be the operator on $\ell^2(\N)$ defined by
$$
S (\xi_0, \xi_1, \xi_2, \xi_3, \xi_4, \hdots)
= (a\xi_1, \xi_2, \xi_3, \xi_4, \hdots) .
$$
(it is an example of weighted shift).
Its adjoint is given by
$$
S^* (\xi_0, \xi_1, \xi_2, \xi_3, \xi_4, \hdots)
= (0, a\xi_0, \xi_1, \xi_2, \xi_3, \hdots) .
$$
We denote by $P_0$ the orthogonal projection of $\ell^2(\N)$ onto the line generated by
$(1, 0, 0, 0, 0, \hdots)$ and by $1$ the identity operator on $\ell^2(\N)$.
We denote by $D_0(1)$ the closed disc $\{ z \in \C \sca \vert z \vert \le 1 \}$
of radius $1$ centred at the origin of $\C$.
\par

The following lemma and its proof are inspired by Section 2 in \cite{Szwa--88}.

\begin{lem}
\label{LemmaSpecwshift}
The notation is as above.
\begin{enumerate}[label=(\arabic*)]
\item\label{1DELemmaSpecwshift}
For $n \ge 1$, we have $\Vert S^n \Vert = \max \{ a, 1 \}$.
\item\label{2DELemmaSpecwshift}
The spectral radius of $S$ is $1$.
\item\label{3DELemmaSpecwshift}
Any $z \in \C$ such that $\vert z \vert < 1$ is an eigenvalue of $S$.
\item\label{4DELemmaSpecwshift}
The spectrum of $S$ is the closed disc $D_0(1)$.
\item\label{5DELemmaSpecwshift}
For $z \in \C$ such that $\vert z \vert \ge 1$, the number $z$ is not an eigenvalue of $S$
(hence the inequality of~\ref{3DELemmaSpecwshift} is sharp).
\item\label{6DELemmaSpecwshift}
For $z \in \C$ such that $\vert z \vert = 1$, the operator $z-S$ is not surjective.
\item\label{7DELemmaSpecwshift}
For $z \in \C$ such that $0 < \vert z \vert < 1$, the operators $1 - zS$ and $1 - zS^*$
are invertible.
\end{enumerate}
\end{lem}

\begin{proof}
\ref{1DELemmaSpecwshift}
On the one hand, for $n \ge 1$, we have $S^n \delta_n = a \delta_0$.
Since $\Vert \delta_n \Vert = \Vert \delta_0 \Vert = 1$,
this implies
$$
a = \langle S^n \delta_n \sca \delta_0 \rangle \le \Vert S^n \Vert
\hskip.5cm \text{and} \hskip.5cm
1 = \langle S^n \delta_{n+1} \sca \delta_1 \rangle \le \Vert S^n \Vert ,
$$
hence $\Vert S^n \Vert \ge \max \{ a, 1 \}$.
On the other hand, for all $\xi$ in $\ell^2(\N)$, we have
$$
\begin{aligned}
\Vert S^n \xi \Vert^2
&= \vert a \xi_n \vert^2 + \vert \xi_{n+1} \vert^2 + \vert \xi_{n+2} \vert^2 + \cdots
= a^2 \vert \xi_n \vert^2 + \vert \xi_{n+1} \vert^2 + \vert \xi_{n+2} \vert^2 + \cdots
\\
&\le \big( \max \{ a, 1 \} \big)^2 \sum_{j=n}^\infty \vert \xi_j \vert^2 
\le \big( \max \{ a, 1 \} \big)^2 \Vert \xi \Vert^2 ,
\end{aligned}
$$
hence $\Vert S^n \Vert \le \max \{ a, 1 \}$.

\vskip.2cm

\ref{2DELemmaSpecwshift}
This follows from~\ref{1DELemmaSpecwshift} because,
as for any bounded linear operator on a Banach space,
the spectral radius of $S$
is given by $\lim_{n \to \infty} \Vert S^n \Vert^{1/n}$.

\vskip.2cm

\ref{3DELemmaSpecwshift}
Set $\xi = (a, z, z^2, z^3, \hdots)$.
Then $\xi \in \ell^2(\N)$ and $\xi \ne 0$,
because $0 < \Vert \xi \Vert^2 = a^2 + \sum_{n=1}^\infty \vert z \vert^{2n} < \infty$.
It is straightforward to check that $S \xi = z \xi$.

\vskip.2cm

\ref{4DELemmaSpecwshift}
By~\ref{2DELemmaSpecwshift}, the spectrum of $S$ is contained in $D_0(1)$.
By~\ref{3DELemmaSpecwshift}, it contains the open disc of radius $1$ around the origin,
and therefore also its closure $D_0(1)$.

\vskip.2cm

\ref{5DELemmaSpecwshift}
Let $z \in \C$, $\vert z \vert \ge 1$, and
$\xi \in \ell^2(\N)$ be such that $S \xi = z \xi$, namely such that
$$
S\xi = (a\xi_1, \xi_2, \hdots, \xi_{n+1}. \hdots)
= z\xi = (z\xi_0, z\xi_1, \hdots, z\xi_n, \hdots) .
$$
Then $\xi_1 = \frac{z}{a} \xi_0$ and $\xi_{n+1} = z \xi_n$ for $n \ge 1$,
hence $\vert \xi_1 \vert \ge \frac{1}{a} \vert \xi_0 \vert$
and $\vert \xi_{n+1} \vert \ge \vert \xi_n \vert$ for $n \ge 1$.
Since $\sum_{n=0}^\infty \vert \xi_n \vert^2 < \infty$, this implies that $\xi_n = 0$ for all $n \ge 0$,
i.e., $\xi = 0$, hence $z$ is not an eigenvalue of $S$.

\vskip.2cm

\ref{6DELemmaSpecwshift}
When $\vert z \vert = 1$, the operator $z-S$ is injective by~\ref{5DELemmaSpecwshift}
and non invertible by~\ref{4DELemmaSpecwshift}, hence it is not surjective.

\vskip.2cm

\ref{7DELemmaSpecwshift}
If $0 < \vert z \vert < 1$,
then $z^{-1}$ is not in the spectrum of $S$ by~\ref{4DELemmaSpecwshift},
i.e., $S - z^{-1}$ is invertible, and therefore $1 - zS = -z(S-z^{-1})$ is invertible.
Similarly $(\overline{z})^{-1}$ is not in the spectrum of $S$
and $1 - \overline{z}S$ is invertible,
so that $1 - zS^* = (1 - \overline{z}S)^*$ is invertible.
\end{proof}

\begin{lem}
\label{lemmaJ}
The notation is as above.
\begin{enumerate}[label=(\arabic*)]
\item\label{1DElemmaJ}
We have
$J_a = S + S^*$, $SP_0 = 0$, and $SS^* = 1 + (a^2 - 1)P_0$.
\item\label{2DElemmaJ}
For all $z \in \C^*$, we have
$$
\frac{1}{z} (1 - zS) \hskip.1cm \big[ (1 - z^2(a^2 - 1)) P_0 + 1 - P_0 \big] \hskip.1cm (1 - zS^*)
= \frac{1}{z} + z - J_a .
$$
\item\label{3DElemmaJ}
For $z \in \C^*$ such that $\vert z \vert = 1$,
the operator $\frac{1}{z} + z - J_a$ is not invertible.
\end{enumerate}
We assume moreover that $a \le \sqrt{2}$.
\begin{enumerate}[label=(\arabic*)]
\addtocounter{enumi}{3}
\item\label{4DElemmaJ}
For $z \in \C^*$, the operator $\frac{1}{z} + z - J_a$ is invertible
if and only if $\vert z \vert \ne 1$.
\item\label{5DElemmaJ}
The spectrum of $J_a$ is $\mathopen[ -2, 2 \mathclose]$
and its norm is $\Vert J_a \Vert = 2$.
\end{enumerate}
\end{lem}

\begin{proof}
The equalities of Claim \ref{1DElemmaJ} are straightforward.
Using them, we show~\ref{2DElemmaJ}:
$$
\begin{aligned}
&\frac{1}{z} (1 - zS) \big[ (1 - z^2(a^2 - 1)) P_0 + 1 - P_0 \big] (1 - zS^*)
\\
& = \hskip.1cm
\frac{1}{z} (1 - zS) \big[ - z^2(a^2 - 1) P_0 + 1 \big] (1 - zS^*)
\\
& = \hskip.1cm
\frac{1}{z} \big[ -z^2 (a^2 - 1) P_0 + 1 + z^3 (a^2 - 1) S P_0 - zS \big] (1 - zS^*)
\\
& = \hskip.1cm
\frac{1}{z} \big[ -z^2 (a^2 - 1) P_0 + 1 - zS \big] (1 - zS^*) 
\\
& = \hskip.1cm
\frac{1}{z} \big[ -z^2 (a^2 - 1) P_0 + 1 - zS + z^3 (a^2 - 1) P_0 S^* - zS^* + z^2 S S^* \big]
\\
& = \hskip.1cm
\frac{1}{z} \big[ -z^2 (a^2 - 1) P_0 + 1 - z J_a + z^2 \big( 1 + (a^2 - 1)P_0 \big) \big]
\\
& = \hskip.1cm
\frac{1}{z} \big[ 1 - z J_a + z^2 ) \big]
\hskip.1cm = \hskip.1cm
\frac{1}{z} + z - J_a .
\end{aligned}
$$

\vskip.2cm

For~\ref{3DElemmaJ}, let $z \in \C^*$ be such that $\vert z \vert = 1$.
The operator $1 - zS = z ( z^{-1} - S)$
is not surjective by~\ref{6DELemmaSpecwshift} of Lemma~\ref{LemmaSpecwshift},
and~\ref{2DElemmaJ} implies that $\frac{1}{z} + z - J_0$
is not surjective, hence in particular is not invertible.

\vskip.2cm

\ref{4DElemmaJ}
When $a \le \sqrt 2$, it remains to show that,
for $z \in \C^*$ such that $\vert z \vert \ne 1$,
the operator $\frac{1}{z} + z - J_a$ is invertible.
Since $z + \frac{1}{z}$ is invariant by the change
$z \mapsto \frac{1}{z}$,
we may assume that $\vert z \vert < 1$.
The operators $1-zS$ and $1 - zS^*$ are invertible by~\ref{7DELemmaSpecwshift}
of Lemma~\ref{LemmaSpecwshift}.
Since $a \le \sqrt{2}$, we have $\vert a^2 - 1 \vert \le 1$,
hence $\vert z^2(a^2 - 1) \vert < 1$ and $1 - z^2 (a^2 - 1) \ne 0$,
so that the operator $(1 - z^2(a^2 - 1)) P_0 + 1 - P_0$ is invertible.
It follows from~\ref{2DElemmaJ} that the operator $\frac{1}{z} + z - J_a$ is invertible.

\vskip.2cm

\ref{5DElemmaJ}
For $z \in \C^*$, we have $\vert z \vert = 1$
if and only if $\frac{1}{z} + z \in \mathopen[ -2, 2 \mathclose]$.
Thus $\lambda - J_a$ is invertible if and only if
$\lambda \notin \mathopen[ -2, 2 \mathclose]$.
Claim~\ref{5DElemmaJ} is proved.
\end{proof}

\begin{lem}
\label{LemmaVp}
Let $a > 0$, let $J_a$ be as above, let $\lambda \in \R$,
and let $\xi \in \ell^2(\N)$ be such that $J_a \xi = \lambda \xi$.
\begin{enumerate}[label=(\arabic*)]
\item\label{1DELemmaVp}
If $\vert \lambda \vert \le 2$, then $\xi = 0$.
\item\label{2DELemmaVp}
If $\vert \lambda \vert > 2$ and $\xi \ne 0$,
then $a > \sqrt 2$ and $\lambda = \pm \frac{ a^2 }{ \sqrt{a^2 - 1} }$.
\end{enumerate}
\end{lem}

\begin{proof}
The space $\ell^2(\N)$ is naturally a subspace
of the space $\C^\N$ of all sequences of complex numbers,
and $J_a$ has a natural extension to a linear operator on $\C^\N$.
Let $\lambda \in \R$ and $\xi \in \C^N$ be such that $J_a \xi = \lambda \xi$,
i.e., such that
$$
\begin{aligned}
(I \hskip.2cm \text{for} \hskip.2cm (J_a \xi)_0 = \lambda \xi_0)
\hskip2.7cm a \xi_1 &= \lambda \xi_0,
\\
(II \hskip.2cm \text{for} \hskip.2cm (J_a \xi)_1 = \lambda \xi_1)
\hskip1.8cm a \xi_0 + \xi_2 &= \lambda \xi_1,
\\
(III \hskip.2cm \text{for} \hskip.2cm (J_a \xi)_n = \lambda \xi_n)
\hskip1.1cm \xi_{n-1} + \xi_{n+1} & = \lambda \xi_n
\hskip.5cm \text{for all} \hskip.2cm
n \ge 2.
\end{aligned}
$$
The characteristic equation $r^2 - \lambda r + 1 = 0$
of the difference equation (III) has roots $\frac{1}{2}(\lambda \pm \sqrt{\lambda^2 - 4})$,
so that the general solution of (III) is, for all $n \ge 1$,
$$
\begin{aligned}
(i) \hskip.5cm
&\xi_n = K_1 \left( \frac{1}{2}(\lambda + \sqrt{\lambda^2 - 4}) \right)^n
+ K_2 \left( \frac{1}{2}(\lambda - \sqrt{\lambda^2 - 4}) \right)^n
\hskip.2cm \text{when} \hskip.2cm
\vert \lambda \vert \ne 2,
\\
(ii) \hskip.4cm
&\xi_n = K_1 + K_2 n
\hskip.2cm \text{when} \hskip.2cm
\lambda = 2,
\\
(iii) \hskip.4cm
&\xi_n = (-1)^n (K_1 + K_2 n)
\hskip.2cm \text{when} \hskip.2cm
\lambda = - 2 ,
\end{aligned}
$$
where $K_1, K_2$ are constants.
Equation (I) provides the value of $\xi_0$
and Equation (II) provides a condition involving $\lambda$ and $a$.

\vskip.2cm

\emph{Proof of~\ref{1DELemmaVp}.}
When $\vert \lambda \vert < 2$,
the roots $\frac{1}{2} (\lambda \pm i \sqrt{4 - \lambda^2})$
have absolute value~$1$,
so that the solution $\xi$ given by (i) is in $\ell^2(\N)$ if and only if $K_1 = K_2 = 0$.
Similarly, when $\vert \lambda \vert = 2$,
the only solution of $J_a \xi = \lambda \xi$ in $\ell^2(\N)$ is the zero vector.

\vskip.2cm

\emph{Proof of~\ref{2DELemmaVp}.}
Suppose first that $\lambda > 2$.
Since $\frac{1}{2} (\lambda + \sqrt{\lambda^2 - 4}) > 1$,
any solution $0 \ne \xi = (\xi_n)_{n \in \N}$ in $\ell^2(\N)$ of $J_a \xi = \lambda \xi$
is such that $\xi_n = K \left( \frac{1}{2}(\lambda - \sqrt{\lambda^2 - 4}) \right)^n$
for some constant $K \ne 0$ and for all $n \ge 1$.
Upon multiplying $\xi$ by a constant, we can assume that $K = 1$.
Equation (I) implies that
$\xi_0 = \frac{a}{2 \lambda} (\lambda - \sqrt{\lambda^2 - 4})$
and Equation~(II) implies successively
$$
\begin{aligned}
&
\frac{ a^2 }{ 2 \lambda } (\lambda - \sqrt{\lambda^2 - 4})
+ \frac{1}{4} (\lambda - \sqrt{\lambda^2 - 4})^2
= \frac{\lambda}{2} (\lambda - \sqrt{\lambda^2 - 4}) ,
\\ &
\frac{ a^2 }{2} - \frac{ a^2 }{ 2 \lambda} \sqrt{ \lambda^2 - 4} - 1 = 0 ,
\hskip.5cm
(a^2 - 2) = \frac{ a^2 }{ \lambda} \sqrt{ \lambda^2 - 4} ,
\hskip.5cm
\text{hence} \hskip.2cm a^2 \ge 2
\\ &
a^4 \lambda^2 - 4 a^2 \lambda^2 + 4 \lambda^2 = a^4 \lambda^2 - 4 a^4
\hskip.5cm \text{hence} \hskip.2cm
\lambda = \frac{ a^2 }{ \sqrt{ a^2-1} } .
\end{aligned}
$$
Suppose now that $\lambda < -2$.
A solution $0 \ne \xi \in \ell^2(\N)$ of $J_a \xi = \lambda \xi$
satisfies $\xi_n = K \left( \frac{1}{2}(\lambda + \sqrt{\lambda^2 - 4}) \right)^n$
for some constant $K \ne 0$ and for all $n \ge 1$.
As above, it follows that $a^2 > 2$ and that $\lambda = - a^2 / \sqrt{ a^2 - 1}$.
\end{proof}

\begin{prop}
\label{PropSpectrumJa}
Let $a > 0$ and let $J_a$ be the Jacobi operator as above.
\begin{enumerate}[label=(\arabic*)]
\item\label{1DEPropSpectrumJa}
When $a \le \sqrt 2$, the operator $J_a$ has norm $2$,
\newline
spectrum $\mathopen[ -2 , 2 \mathclose]$,
\newline
and it has no eigenvalue.
\item\label{2DEPropSpectrumJa}
When $a > \sqrt 2$, the operator $J_a$ has norm $\frac{ a^2 }{ \sqrt{a^2 - 1} }$,
\newline
spectrum $\big\{ - \frac{ a^2 }{ \sqrt{a^2 - 1} } \big\} \cup \mathopen[ -2 , 2 \mathclose]
\cup \big\{ \frac{ a^2 }{ \sqrt{a^2 - 1} } \big\}$,
\newline
it has two simple eigenvalues $\pm \frac{ a^2 }{ \sqrt{a^2 - 1} }$,
and no other eigenvalue.
\end{enumerate}
\end{prop}

For a third information concerning $J_a$, see Proposition~\ref{PropCyclicJa} below.

\begin{proof}
Claim~\ref{1DEPropSpectrumJa} is contained in
Lemmas~\ref{lemmaJ}~\ref{5DElemmaJ} and~\ref{LemmaVp}~\ref{1DELemmaVp}.
\par

For~\ref{2DEPropSpectrumJa}, recall first that,
for a bounded self-adjoint operator $X$
on an infinite dimensional Hilbert space $\Hi$,
the \emph{discrete spectrum} $\Sigma_{{\rm disc}}(X)$
is the set of isolated points of the spectrum which are eigenvalues of finite multiplicity,
and the \emph{essential spectrum} $\Sigma_{{\rm ess}} (X)$
is the complement $\Sigma (X) \smallsetminus \Sigma_{{\rm disc}}(X)$;
the essential spectrum is a nonempty closed subset of the spectrum.
It is a theorem of Weyl that, for a compact self-adjoint operator $K$ on $\Hi$,
the essential spectra of $X$ and $X+K$ are the same
\cite[Theorem 3.14.1]{Sim4--15}.
\par

Consider the free Jacobi matrix $J_1$.
By Lemma~\ref{lemmaJ}, we know that its spectrum is $\mathopen[ -2 , 2 \mathclose]$,
so that its essential spectrum coincides with its spectrum.
For $a > 0$, the operator $J_a$ is a compact perturbation of $J_1$
(indeed a finite rank perturbation).
By Weyl's theorem,
the essential spectrum of $J_a$ is also $\mathopen[ -2 , 2 \mathclose]$.
It follows that the spectrum of $J_a$ is the union of $\mathopen[ -2 , 2 \mathclose]$
and of the eigenvalues of $J_a$, which are necessarily of finite multiplicity.
Claim~\ref{2DEPropSpectrumJa} now follows from
Lemma~\ref{LemmaVp}~\ref{2DELemmaVp}.
\end{proof}

Before showing what we need on dominant vectors for $J_a$,
we recall some general facts on orthogonal polynomials and Jacobi operators.
Consider an infinite sequence $a_* = (a_n)_{n \ge 0}$ of positive real numbers such that
$$
\sup_{n \ge 0} a_n < \infty .
\leqno{(Fa)}
$$
(Later we will need the case $a_* = (a, 1, 1, 1, \hdots)$
of the subdiagonal entries of the matrix $J_a$.)
Define a sequence $(Q_n)_{n \ge 0}$ of polynomials with real coefficients
and positive leading coefficients by $Q_0(t) = 1$ and
$$
\begin{aligned}
t Q_0 (t) &=
\hskip2.7cm a_0 Q_1 (t)
\\
t Q_n (t) &= a_{n-1} Q_{n-1} (t) + a_n Q_{n+1} (t)
\hskip.5cm \text{for all} \hskip.2cm
n \ge 1 .
\end{aligned}
\leqno{(Rec})
$$
The main claim of the following proposition
is the particular case needed below of a theorem of Favard.
[In the general case, there is moreover a bounded sequence $(b_n)_{n \ge 0}$
of real numbers, and the recurrence relation is of the form
$tQ_n (t) = a_{n-1} Q_{n-1} (t) + b_n Q_n(t) + a_n Q_{n+1} (t)$.]
Here, a measure on $\R$ is \emph{nontrivial} if its closed support
is not a finite set of points.
For a measure $\mu$ on $\R$ which is nonzero
and which has compact closed support,
we denote by ${\rm minsupp}(\mu)$ the largest number
and by ${\rm maxsupp}(\mu)$
the smallest number
such that the closed support of $\mu$ is contained in the closed interval
$\mathopen[ {\rm minsupp}(\mu) , \hskip.1cm {\rm maxsupp}(\mu) \mathclose]$.

\begin{prop}[Favard]
\label{PropFavard}
Let $(a_n)_{n \ge 0}$ and $(Q_n)_{n \ge 0}$ be as above,
satisfying respectively (Fa) and (Rec).
\par

Then there exists a nontrivial finite positive measure $\mu$ on $\R$,
with closed support contained in 
$\mathopen[ - \sup_{n \ge 0} a_n, \sup_{n \ge 0} a_n \mathclose]$,
such that the $Q_n$~'s are orthonormal polynomials with respect to $\mu$,
i.e., such that
$$
\int_\R Q_m(t) Q_n(t) d\mu(t) = \left\{
\begin{aligned}
1 \hskip.2cm &\text{if} \hskip.2cm m = n
\\
0 \hskip.2cm &\text{if} \hskip.2cm m \ne n .
\end{aligned}
\right.
$$
\par

Moreover, for each $n \ge 1$, the polynomial $Q_n$
has $n$ simple roots in the open interval
$\mathopen] {\rm minsupp}(\mu) , \hskip.1cm {\rm maxsupp}(\mu) \mathclose[$.
\end{prop}

\begin{proof}[Reference for the proof and remark]
For the proof of the claim on $\mu$, which is the Favard Theorem,
we refer to Section~4.1 in \cite{Sim4--15},
or alternatively to \cite[Theorem 1.3.7]{Simo--11}.
\par

For the proof of the fact that the roots of $Q_n$ are all real and simple,
see~\cite[Proposition 4.1.4]{Sim4--15}.
Since we have not found in \cite{Sim4--15} a proof that these roots are all
in $\mathopen] {\rm minsupp}(\mu) , {\rm maxsupp}(\mu) \mathclose[$,
here is the argument of Szeg\"o's book on orthogonal polynomials.
\par

Let $n \ge 1$, and let $\ell$ denote
the number of roots $t_1 < t_2 < \cdots < t_\ell$ of $Q_n$
in the open interval $\mathopen] {\rm minsupp}(\mu) , {\rm maxsupp}(\mu) \mathclose[$.
Define a polynomial $R_n$ by
$R_n(t) = Q_n(t) \prod_{j=1}^\ell (t-t_j)$.
Assume by contradiction that $\ell < n$.
On the one hand, $R_n$ does not change sign on this interval,
hence $\int_{{\rm minsupp}(\mu)}^{{\rm maxsupp}(\mu)} R_n(t) d\mu(t) \ne 0$.
On the other hand, $Q_n$ is orthogonal to any polynomial of degree $< n$,
hence
\newline
$\int_{{\rm minsupp}(\mu)}^{{\rm maxsupp}(\mu)} R_n(t) d\mu(t) = 0$.
This is absurd, and therefore $\ell = n$.
\end{proof}

Together with the infinite sequence of positive numbers $a_* = (a_n)_{n \ge 0}$
satisfying (Fa) above, consider now
the Hilbert space $\ell^2(\N)$,
its standard orthonormal basis $\delta_* = (\delta_n)_{n \ge 0}$,
and the bounded self-adjoint operator $J = J_{a_*}$ defined by
$$
\begin{aligned}
J_{a_*} \delta_0&=
\hskip2.1cm a_0 \delta_1
\\
J_{a_*} \delta_n&= a_{n-1} \delta_{n-1} + a_n\delta_{n+1}
\hskip.5cm \text{for all} \hskip.2cm
n \ge 1 ,
\end{aligned}
$$
of matrix
$$
J_{a_*} = \begin{pmatrix}
0 & a_0 & 0 & 0 & \cdots
\\
a_0 & 0 & a_1 & 0 & \cdots
\\
0 & a_1 & 0 & a_2 & \cdots
\\
0 & 0 & a_2 & 0 & \cdots
\\
\cdots & \cdots & \cdots & \cdots & \ddots
\end{pmatrix}
\leqno{(Jac_{a_*})}
$$
with respect to $\delta_*$.

\begin{prop}
\label{propuniteqJM}
Let $a_*$ and $J_{a_*}$ be as above,
let $\mu$ be the measure of Proposition~\ref{PropFavard},
and let $M_\mu$ be the operator
of multiplication by $t$ on the Hilbert space $L^2(\R, \mu)$,
as in Proposition~\ref{c+m}.
\par

Then the operators $J_{a_*}$ and $M_\mu$
are unitarily equivalent.
\end{prop}

\begin{proof}
The sequence $(Q_n)_{n \ge 0}$ defined above
is an orthonormal basis of the space $L^2(\mu)$.
Define an operator $W \, \colon \ell^2(\N) \to L^2(\R, \mu)$ by
$W \delta_n = Q_n$ for all $n \ge 0$;
note that $W$ is a surjective isometry.
Then
$$
W J_{a_*} W^* = M_\mu .
$$
(A particular case of this appears again
in the proof of Proposition~\ref{PropExRay}.)
\end{proof}

For $n \ge 0$, the vector $\delta_n$ is cyclic for $J_{a_*}$
if and only if the vector $Q_n$ is cyclic for $M_\mu$,
and by Proposition~\ref{c+m} this is the case if and only if
$$
\mu( \{ t \in \Sigma \hskip.1cm : \hskip.1cm Q_n (t) = 0 \} ) = 0 ,
$$
where $\Sigma$ denotes the closed support of $\mu$.
As we do not know whether this is true or not,
we come back to the special case of the sequence
$(a, 1, 1, 1, \hdots)$.

\begin{prop}
\label{PropCyclicJa}
Let $a > 0$ and let $J_a$ be the Jacobi operator defined on $\ell^2(\N)$
by the matrix of Definition~\ref{defJa}.
\par
Then, for all $n \ge 0$, the standard basis vector $\delta_n$
is cyclic for $J_a$.
\end{prop}

\begin{proof}
By Proposition~\ref{PropFavard},
for any $n \ge 0$, the zero set of the polynomial $Q_n$
is a finite subset of $\mathopen] {\rm minsupp}(\mu) , {\rm maxsupp}(\mu) \mathclose[$.
By Proposition~\ref{PropSpectrumJa},
$$
\mathopen] {\rm minsupp}(\mu) , {\rm maxsupp}(\mu) \mathclose[
= \left\{ \begin{aligned}
\mathopen] - 2 , 2 \mathclose[ \hskip1.4cm
\hskip.5cm \text{if} \hskip.2cm &a \le \sqrt 2 
\\
\mathopen\big] - \frac{ a^2 }{ \sqrt{a^2 - 1} } , \frac{ a^2 }{ \sqrt{a^2 - 1} } \mathclose\big[
\hskip.5cm \text{if} \hskip.2cm &a > \sqrt 2 
\end{aligned} \right.
$$
does not contain any eigenvalue of $J_a$,
hence the $\mu$ measure of any finite subset of this interval is $0$,
hence $\mu( \{ t \in \Sigma \hskip.1cm : \hskip.1cm Q_n (t) = 0 \} ) = 0$.
By Proposition~\ref{c+m}, it follows that $Q_n$
is a cyclic vector for the multiplication operator $M_\mu$.
\par

Since $W \delta_n = Q_n$ and $W J_a W^* = M_\mu$,
the basis vector $\delta_n$ is also a cyclic vector for the Jacobi operator $J_a$.
\end{proof}

\section{Further examples: the infinite ray, the graph $D_\infty$,
\\
and the stars}
\label{SectionFE}

\begin{prop}[\textbf{in the infinite ray, all vertices are cyclic, and therefore dominant}]
\label{PropExRay}
Let $P$ be the one-way infinite path, or infinite ray,
with vertex set $\N = \{0, 1, 2, 3, \hdots\}$
and edge set $E = \{ \{j, j+1\} : j \in \N \}$.
\par

Then all vertices of $P$ are cyclic.
\end{prop}

\noindent \emph{Remark.}
The proof will show more:
any non-zero finite linear combination of the vectors $\delta_j$ (with $j \in \N$)
is a cyclic vector for the adjacency operator of $P$.

Before this, we show the claim of Proposition~\ref{PropExRay}
for the vertices $j=0$, $j=1$, and $j=2$.

\begin{proof}[Proof for the three first vertices]
The adjacency operator $A$ of $P$ acts on $\ell^2(\N)$ by
$$
(A \xi)(n) = \xi(n-1) + \xi(n+1)
\hskip.5cm \text{for all} \hskip.2cm
\xi \in \ell^2(\N)
\hskip.2cm \text{and} \hskip.2cm
n \in \N ,
$$
where $\xi(-1)$ should be read as~$0$.
Equivalently, it acts on the canonical orthonormal basis
$(\delta_n)_{n \in \N}$ of $\ell^2(\N)$ by
\begin{equation}
\label{EqA=Jac1}
A \delta_n = \delta_{n-1} + \delta_{n+1}
\hskip.5cm \text{for all} \hskip.2cm
n \ge 0 ,
\end{equation}
where $\delta_{-1}$ should be read as $0$.
\par

We will show that the vector $\delta_j$ for $j = 0, 1, 2$ is cyclic for $A$,
namely that the closed linear span $\Hi_j$ of $\{A^n \delta_j)_{n=0}^\infty$
is the whole of $\ell^2(\N)$.
\par

Consider first the vector $\delta_0$ in $\ell^2(\N)$.
The space $\Hi_0$ contains $\delta_0$,
hence $A \delta_0 = \delta_1$,
hence $A \delta_1 = \delta_0 + \delta_2$ and therefore also $\delta_2$,
hence $A \delta_2 = \delta_1 + \delta_3$ and therefore also $\delta_3$,
and so on.
Thus $\Hi_0$ contains all $\delta_j$ and therefore coincides with $\ell^2(\N)$.
\par

The space $\Hi_1$ contains $\delta_1$,
hence $A \delta_1 = \delta_0 + \delta_2$,
hence $A (\delta_0 + \delta_2) - 2 \delta_1 = \delta_3$,
hence $A \delta_3 = \delta_2 + \delta_4$, 
and so on.
Thus $\Hi_1$ contains $\delta_{2n+1}$ and $\delta_{2n} + \delta_{2n+2}$
for all $n \ge 0$.
To show that $\Hi_1 = \ell^2(\N)$, it suffices to show that
any $\eta = (\eta_j)_{j=0}^\infty \in \ell^2(\N)$ orthogonal to $\Hi_1$ is zero.
On the one hand
$\langle \eta \sca \delta_{2n+1} \rangle = \eta_{2n+1} = 0$ for all $n \in \N$.
On the other hand
$\langle \eta \sca \delta_{2n} + \delta_{2n+2} \rangle = \eta_{2n} + \eta_{2n+2} = 0$
and therefore $\vert \eta_{2n} \vert = \vert \eta_{2n+2} \vert$
for all $n \in \N$; since $\sum_{n=0}^\infty \vert \eta_{2n} \vert^2 < \infty$,
this implies $\eta_{2n} = 0$ for all $n \in \N$.
Hence $\eta = 0$, as was to be shown.
\par

The space $\Hi_2$ contains $\delta_2$,
hence $A\delta_2 = \delta_1 + \delta_3$,
hence $A(\delta_1 + \delta_3) - 2 \delta_2 = \delta_0 + \delta_4$,
hence $A(\delta_0 + \delta_4) - (\delta_1 + \delta_3) = \delta_5$,
hence $A\delta_5 = \delta_4 + \delta_6$,
hence $A(\delta_4 + \delta_6) - 2 \delta_5 = \delta_3 + \delta_7$,
hence $A(\delta_3 + \delta_7) - (\delta_4 + \delta_6) - \delta_2 = \delta_8$,
and so on.
Thus $\Hi_2$ contains
$\delta_{3n+2}$, $\delta_{3n+1}+\delta_{3n+3}$, and $\delta_{3n}+\delta_{3n+4}$
for all $n \in \N$.
It follows that, for a vector $(\eta_j)_{j=0}^\infty \in \ell^2(\N)$ orthogonal to $\Hi_2$,
we have $\eta_{3n+2} = 0$ for all $n \in \N$,
and $\vert \eta_1 \vert = \vert \eta_3 \vert = \vert \eta_7 \vert = \vert \eta_9 \vert = \cdots$,
and $\vert \eta_0 \vert = \vert \eta_4 \vert = \vert \eta_6 \vert = \vert \eta_{10} \vert = \cdots$,
and therefore $\eta_j = 0$ for all $j \in \N$, so that $\Hi_2 = \ell^2(\N)$.
\par

It should be elementary to continue and show that $\Hi_j = \ell^2(\N)$ for all $j \ge 0$,
but we find
the combinatorics more and more complicated,
and we will provide below a proof of a different style.
\end{proof}

\begin{proof}[Proof of Proposition~\ref{PropExRay}]
We first describe how $A$
is unitarily equivalent to a multiplication operator
closely related to the sequence of Chebyshev polynomials.
\par

Consider the sequence $\left( P_n \right)_{n=0}^\infty$ of functions
defined on the interval $\mathopen[ -2, 2 \mathclose]$ of the real line by
$$
P_n( 2 \cos (\theta) ) = \frac{ \sin ( (n+1)\theta ) }{ \sin ( \theta ) } .
$$
Note that $P_0 (t) = 1$, $P_1(t) = t$, $P_2(t) = t^2-1$,
for all $t \in \mathopen[ -2, 2 \mathclose]$.
Define $P_{-1}$ to be the zero function.
From the trigonometric formula
$$
\sin( (n+1)\theta ) + \sin( (n-1)\theta ) = 2 \cos (\theta) \sin( n\theta ) ,
$$
it follows that
\begin{equation}
\label{EqRecCheb}
t P_{n-1} (t) = P_n (t) + P_{n-2} (t)
\hskip.5cm \text{for all} \hskip.2cm
n \ge 1
\hskip.2cm \text{and} \hskip.2cm
t \in \mathopen[ -2, 2 \mathclose] .
\end{equation}
This implies, by induction on $n$, that $P_n$ is a polynomial,
of the form $P_n(t) = t^n + (\text{lower order terms})$, for all $n \ge 0$.
\par

Define a probability measure $\mu$ on $\mathopen[ -2, 2 \mathclose]$ by
\begin{equation}
\label{EqMesureJacobi1}
d \mu (t) = \frac{1}{2\pi} \sqrt{ 4 - t^2} \hskip.1cm dt
\hskip.5cm \text{for} \hskip.2cm
t \in \mathopen[ -2, 2 \mathclose].
\end{equation}
For $m, n \ge 0$, 
using the change of variables $t = 2 \cos (\theta)$,
we have
$$
\begin{aligned}
&
\int_{-2}^2 P_m (t) P_n (t) d\mu(t)
= \frac{1}{2\pi} \int_{-2}^2 P_m (t) \sqrt{ 4 - t^2}
\hskip.1cm P_n (t) \sqrt{ 4 - t^2} \hskip.1cm \frac{dt}{ \sqrt{ 4 - t^2} }
\\
& \hskip1cm =
\frac{1}{2\pi} \int_0^\pi P_m(2 \cos (\theta)) \hskip.1cm 2 \sin (\theta)
\hskip.1cm P_n(2 \cos (\theta)) \hskip.1cm 2 \sin (\theta) \hskip.1cm d\theta
\\
& \hskip1cm =
\frac{2}{\pi} \int_0^\pi \sin ((m+1)\theta) \hskip.1cm \sin ((n+1)\theta) \hskip.1cm d\theta
\\
& \hskip1cm =
\frac{1}{\pi} \int_0^\pi \Big[ \cos(((m-n)\theta) - \cos((m+n+2)\theta) \Big] \hskip.1cm d\theta
\\
& \hskip1cm =
0 \hskip.2cm \text{if} \hskip.2cm m \ne n
\hskip.2cm \text{and} \hskip.2cm
1 \hskip.2cm \text{if} \hskip.2cm m = n .
\end{aligned}
$$
It follows that $(P_n)_{n \ge 0}$
is an orthonormal basis of $L^2( \mathopen[ -2, 2 \mathclose], \mu )$.
If $M_\mu$ denotes the operator of multiplication by $t$ on this space,
we have by~\eqref{EqRecCheb} above
\begin{equation}
\label{EqJ1=mult}
M_\mu P_n = P_{n-1} + P_{n+1}
\hskip.5cm \text{for all} \hskip.2cm
n \ge 0 .
\end{equation}
\par

The $P_n$~'s are Chebyshev polynomials, up to a normalization.
More precisely, if $U_n(t)$ denotes
the Chebyshev polynomial of the second kind of degree~$n$,
defined by $U_n (\cos \theta) = \sin ((n+1)\theta) / \sin ( \theta )$,
then $P_n (t) = U_n( t/2)$.
\par

Define now an operator $W \, \colon \ell^2(\N) \to L^2( \mathopen[ -2, 2 \mathclose], \mu )$
by $W(\delta_n) = P_n$ for all $n \ge 0$;
it is a surjective isometry.
Then
$$
W A (\delta_n) = W (\delta_{n-1} + \delta_{n+1})
= P_{n-1} + P_{n+1} = M_\mu(P_n) = M_\mu W (\delta_n)
$$
for all $n \ge 0$,
hence $W A W^* = M_\mu$, i.e., $A$ and $M_\mu$ are unitarily equivalent.
\par

For all $n \ge 0$ we have $\mu ( \{t \in \mathopen[ -2, 2 \mathclose] : P_n (t) = 0\} ) = 0$,
because $P_n$ has a finite set of roots.
By Proposition~\ref{c+m},
the vector $P_n \in L^2( \mathopen[ -2, 2 \mathclose], \mu )$
is dominant for~$M_\mu$, and also cyclic for $M_\mu$.
It follows that the vector $\delta_n = W^{-1}(P_n) \in \ell^2(\N)$
is dominant for $A$, and also cyclic for $A$.
\par

Finally about the Remark which follows Proposition~\ref{PropExRay}:
for the same reason, any non-zero polynomial in
$L^2( \mathopen[ -2, 2 \mathclose], \mu )$
is a cyclic vector for $M_\mu$,
and therefore any non-zero finite linear combination of the $\delta_j$~'s
is a cyclic vector for $A$.
\end{proof}

\begin{exe}[\textbf{an infinite graph with exactly two dominant vertices}]
\label{Only2dom}
Consider the graph $D_\infty$, obtained from the infinite ray $P$ of Proposition~\ref{PropExRay}
by adding one vertex $0'$ and one edge $\{0',1\}$.
Then $D_\infty$ has two cyclic vertices, which are $0$ and $0'$,
and all other vertices are not cyclic, and not dominant.
\par

To prove that the vertices $0$ and $0'$ are cyclic,
the argument used in the proof of Proposition~\ref{PropExRay}
for the vertex $0$ of $P$ applies with a minor adaptation.
For the other vertices, consider the subspace $\Ki$ of $\ell^2(D_\infty)$
spanned by the vector $\delta_0 - \delta_{0'}$,
which is an eigenvector of eigenvalue $0$.
The hyperplane $\Ki^\perp$ is invariant by the adjacency operator of $D_\infty$,
and contains $\delta_j$ for all $j \ge 1$.
It follows that $\delta_j$ is not a cyclic vector,
i.e., that $j$ is not a cyclic vertex, for all $j \ge 1$.
By Corollary~\ref{c=d}, $j$ is not a dominant vertex for all $j \ge 1$.
\end{exe}

Let $k$ be an integer, $k \ge 2$, and let $S_k = (V_k, E_k)$ be the star with $k$ infinite rays.
The vertex set $V_k$ consists of vertices $v_0$ and $v_{j,n}$
for $j \in \{1, 2, \hdots, k\}$ and $n \in \N_{\ge 1}$.
Whenever $v_{j, 0}$ for some $j$ appears below, it should be read as $v_0$.
The edge set $E_k$ of $S_k$ contains edges $\{ v_{j,n}, v_{j,n+1} \}$
for $j \in \{1, 2, \hdots, k\}$ and $n \in \N$.
Note that the core vertex $v_0$ is of degree $k$,
and that all other vertices in $V_k$ are of degree $2$.
Note also that $S_2$ is the two-way infinite path $L^1$.

\begin{thm}[\textbf{in the infinite stars, all vertices (but possibly the core)
are dominant and are not cyclic}]
\label{exStars}
Let $k \ge 3$ and let $S_k$ be the star with $k$ infinite rays, as above.
\par

In $S_k$, all vertices other than the core are dominant,
and all vertices are non cyclic.
\end{thm}

For $k = 2$, the ``star'' $S_2$ is the infinite line $L^1$,
in which all vertices are dominant,
see Example~\ref{exLatticeTree}.
When $k \ge 3$,
we have not been able to determinate whether the core $v_0$ in $S_k$ is dominant or not.
\par

The proof below uses general facts
on orthogonal polynomials on the real line and related self-adjoint operators
(compare with the proof of Proposition~\ref{PropExRay})
which have been recalled in Section~\ref{S+EofJc}.

\begin{proof}[First part of the proof of Theorem~\ref{exStars}:
all vertices $\ne v_0$ are dominant]
Denote by $\Hi$ the Hilbert space $\ell^2(V_k)$.
The space $\Hi$ has an orthonormal basis consisting of
$$
\delta_0 := \delta_{v_o}
\hskip.5cm \text{and} \hskip.5cm
\delta_{j, n} := \delta_{v_{j, n}}
\hskip.2cm \text{for} \hskip.2cm
j \in \{1, 2, \hdots, k\}
\hskip.2cm \text{and} \hskip.2cm
 n \in \N_{\ge 1}.
$$
The adjacency operator $A$ of $S_k$ acts on $\Hi$ by
$$
A \delta_0 = \sum_{j=1}^k \delta_{j, 1}
\hskip.2cm \text{and} \hskip.2cm
A \delta_{j, n} = \delta_{j, n-1} + \delta_{j, n+1}
\hskip.2cm \text{for} \hskip.2cm
j \in \{1, 2, \hdots, k\}
\hskip.2cm \text{\&} \hskip.2cm
n \in \N_{\ge 1}
\leqno{(*)}
$$
(where $\delta_{j, 0}$ stands for $\delta_0$).
\par

There is a natural action of the cyclic group of order $k$ on $V_k$,
by cyclic permutations of the rays of $S_k$,
and a corresponding orthogonal decomposition of $\Hi$ in isotypical components,
which can be defined as follows.
Set $\omega = \exp( 2 \pi i / k)$. For $\ell \in \{1, 2, \hdots, k-1\}$, define
$$
\Hi_\ell =
\Bigg\{ \xi \in \Hi
\hskip.1cm \Bigg\vert \hskip.1cm
\begin{aligned}
& \xi (v_0) = 0
\hskip.2cm \text{and, for all} \hskip.2cm
n \in \N_{\ge 1},
\\
& \omega^\ell \xi(v_{1,n}) = \cdots = \omega^{j\ell} \xi(v_{j,n}) = \cdots
= \omega^{(k-1)\ell} \xi(v_{k-1,n}) = \xi(v_{k,n})
\end{aligned}
\Bigg\} ;
$$
this subspace of $\Hi$ has an orthonormal basis
$\varepsilon_{\ell,\ast} = \left( \varepsilon_{\ell,n} \right)_{n \ge 1}$
defined by
$$
\varepsilon_{\ell,n} = \frac{1}{\sqrt k} \sum_{j=1}^k \omega^{-j\ell} \delta_{j,n}
\hskip.2cm \text{for all} \hskip.2cm
n \ge 1 .
$$
Define also
$$
\Hi_k =
\Bigg\{ \xi \in \Hi
\hskip.1cm \Bigg\vert \hskip.1cm
\begin{aligned}
& \xi(v_{1,n}) = \cdots = \xi(v_{j,n}) = \cdots = \xi(v_{k-1,n}) = \xi(v_{k,n})
\\
& \text{for all} \hskip.2cm
n \in \N_{\ge 1}
\\
\end{aligned}
\Bigg\} ;
$$
this subspace has an orthonormal basis
$\varepsilon_{k,\ast} = \left( \varepsilon_{k,n} \right)_{n \ge 0}$
defined by
$$
\varepsilon_{k,0} = \delta_0
\hskip.5cm \text{and} \hskip.5cm
\varepsilon_{k,n} = \frac{1}{\sqrt k} \sum_{j=1}^k \delta_{j,n}
\hskip.2cm \text{for all} \hskip.2cm
n \ge 1 .
$$
Observe that
$$
\delta_{j,n} = \frac{1}{\sqrt k} \sum_{\ell=1}^k \omega^{j \ell} \varepsilon_{\ell,n}
\hskip.2cm \text{for all} \hskip.2cm
j \in \{1, \hdots, k\}
\hskip.2cm \text{and} \hskip.2cm
n \ge 1 ;
\leqno{(**)}
$$
indeed
$$
\frac{1}{\sqrt k} \sum_{\ell=1}^k \omega^{j \ell} \varepsilon_{\ell,n}
= \frac{1}{k} \sum_{\ell=1}^k \omega^{j\ell} \sum_{m=1}^k \omega^{-m\ell} \delta_{m,n}
= \frac{1}{k} \sum_{m=1}^k \Big( \sum_{\ell=1}^k \omega^{(j-m)\ell} \Big) \delta_{m,n}
= \delta_{j,n} .
$$
From ($**$), retain that, for all $j \in \{1, \hdots, k\}$ and $n \ge 1$,
the vector $\delta_{j,n}$ is a linear combination
of $\varepsilon_{1,n}, \hdots, \varepsilon_{k,n}$ in which all coefficients are non-zero.
We leave it to the reader to check that the subspaces defined above
are orthogonal with each other,
that they provide an orthogonal decomposition
$$
\Hi = \Hi_1 \oplus \cdots \oplus \Hi_k ,
$$
and that each of them is invariant by $A$;
we denote by $A_j$ the restriction of $A$ to $\Hi_j$,
and we have
$$
A = A_1 \oplus \cdots \oplus A_k .
$$
\par

For $\ell \in \{1, \hdots, k-1\}$, the matrix of the operator $A_\ell$
with respect to the basis $\left( \varepsilon_{\ell,n} \right)_{n \ge 1}$ of $\Hi_\ell$ is
$$
J_1 = \begin{pmatrix}
0 & 1 & 0 & 0 & \cdots
\\
1 & 0 & 1 & 0 & \cdots
\\
0 & 1 & 0 & 1 & \cdots
\\
0 & 0 & 1 & 0 & \cdots
\\
\cdots & \cdots & \cdots & \cdots & \ddots
\end{pmatrix} 
$$
(independent on $\ell$).
It follows from Proposition~\ref{PropExRay} that, for all $n \ge 1$, the vector $\varepsilon_{\ell,n}$
is cyclic for~$A_\ell$ (as well as all its nonzero scalar multiples).
For $\ell = k$, the matrix of the operator $A_k$
with respect to the basis $\left( \varepsilon_{k,n} \right)_{n \ge 0}$ of $\Hi_k$ is
$$
J_{\sqrt k} = \begin{pmatrix}
0 & \sqrt k & 0 & 0 & \cdots
\\
\sqrt k & 0 & 1 & 0 & \cdots
\\
0 & 1 & 0 & 1 & \cdots
\\
0 & 0 & 1 & 0 & \cdots
\\
\cdots & \cdots & \cdots & \cdots & \ddots
\end{pmatrix} .
$$
By Proposition~\ref{PropCyclicJa},
the vector $\varepsilon_{k,0} = \delta_0$ and the vectors $\varepsilon_{k,n}$ for all $n \ge 1$
are cyclic for~$A_k$.
\par

It follows now from Proposition~\ref{Sumdominant} and from ($**$)
that $\delta_{j,n}$ is dominant for $A$ for all $j \in \{1, \hdots, k\}$ and $n \ge 1$.
\end{proof}

\begin{proof}[Second part of the proof of Theorem~\ref{exStars}: no vertex is cyclic]
The vector $\delta_0$ is contained in the subspace $\Hi_k$,
which is a $A$-invariant proper subspace of $\Hi$,
hence the vector $\delta_0$ is not cyclic for $A$, i.e., the vertex $v_0$ is not cyclic in $S_k$.
In the case of the star $S_2$ with two branches, this argument
applies to all vertices of $S_2 = L^1$.
\par

Assume from now on that $k \ge 3$.
Set
$$
\Ki_1 = \left\{ \xi \in \Hi
\hskip.1cm \big\vert \hskip.1cm
\xi(v_{2, n}) = \xi(v_{3,n}) = \cdots = \xi(v_{k, n})
\hskip.2cm \text{for all} \hskip.2cm
n \in \N_{\ge 1} \right\} .
$$
It is a closed subspace of $\Hi$ which is proper (because $k \ge 3$),
contains $\delta_{1,n}$ for all $n \in \N$,
and which is invariant by $A$ (see ($*$)).
It follows that $\delta_{1,n}$ ($n \ge 0$) is not a cyclic vector for $A$.
With minor modifications,
this argument shows that $\delta_{j,n}$ is not a cyclic vector for $A$,
for all $j \in \{1, \hdots, k\}$ and all $n \ge 0$.
\par

This ends the proof of Theorem~\ref{exStars}.
\end{proof}

\section{Markov operators and graphs without $M$-dominant vertices}
\label{sectionMarkov}

Let $G = (V, E)$ be a graph in which all vertices have finitely many neighbours.
For $u \in V$, denote by $\deg (u)$ the number of its neighbour,
and set $\deg_\sharp (u) = \max \{ \deg (u), 1 \}$.
Rather than $\ell^2(V)$ and $A$, consider the Hilbert space
$$
\ell^2_\sharp(V) = \left\{ \eta \, \colon V \to \C \mid
\sum_{u \in V} \vert \eta (u) \vert^2 \deg_\sharp (u) < \infty \right\} ,
$$
with scalar product denoted by
$\langle \eta_1 \mid \eta_2 \rangle_\sharp
= \sum_{u \in V} \eta_1(u) \overline{\eta_2(u)} \deg_\sharp (u)$.
For $v \in V$, we denote by $\delta_v \in \ell^2_\sharp(V)$
the characteristic function of the vertex $v$,
and by $\varepsilon_v \in \ell^2_\sharp(V)$ the vector $\frac{ 1 }{ \sqrt{ \deg_\sharp(V) } } \delta_v$;
the family $(\varepsilon_v)_{v \in V}$ is an orthonormal basis of $\ell^2_\sharp(V)$.
The \textbf{Markov operator} $M$ of $G$ defined by
$$
(M \eta)(u) = \frac{1}{\deg (u)} 
\sum_{v \sim u} \eta (v)\hskip.5cm \text{for all} \hskip.2cm
\eta \in \ell^2_\sharp(V)
\hskip.2cm \text{and} \hskip.2cm
u \in V ,
$$
where $\sum_{v \sim u}$ is a summation over all vertices $v$ which are neighbours of $u$.
(If $\deg (u) = 0$, the summation is on an empty set of $v$~'s and $(M\eta)(u) = 0$.)
It is elementary to check that
$\Vert M \Vert \le 1$ and that $M$ is self-adjoint on $\ell^2_\sharp(V)$
(see \cite[\S~3.2]{Kowa--19}).
We have
$$
M \delta_w = \sum_{u \sim w} \frac{1}{ \deg (u) } \delta_u
\hskip.5cm \text{and} \hskip.5cm
M \varepsilon_w = \sum_{u \sim w} \frac{ 1 }{ \sqrt{ \deg (u) \deg (w) } } \varepsilon_u ,
$$
so that the matrix $\left( M_{v,w} \right)_{v,w \in V}$ of $M$
with respect to the basis $(\varepsilon_v)_{v \in V}$ is given by
$$
M_{v,w} =
\langle M \varepsilon_w \mid \varepsilon_v \rangle_\sharp = \left\{
\begin{aligned}
\frac{ 1 }{ \sqrt{ \deg (v) \deg (w) } }
\hskip.5cm &\text{if} \hskip.2cm
v,w
\hskip.2cm \text{are neighbours in} \hskip.2cm
V 
\\
0 \hskip1.3cm
\hskip.5cm &\text{otherwise}
\end{aligned}
\right.
$$
(we think of $v$ as a row index and of $w$ as a column index for the matrix of $M$).
Observe that $\frac{ \sqrt{ \deg (w) } }{ \sqrt{ \deg (v) } } M_{v,w}$,
which is $\frac{ 1 }{ \deg (v) }$ if $w$ is a neighbour of $v$ and $0$ if not,
is the probability that a random walk starting at $v$ is at $w$ after one step;
more generally
$$
\frac{ \sqrt{ \deg (w) } }{ \sqrt{ \deg (v) } } (M^n)_{v,w} =
\frac{ \sqrt{ \deg (w) } }{ \sqrt{ \deg (v) } } \langle M^n \varepsilon_w \mid \varepsilon_v \rangle_\sharp =
p^{(n)}_{v,w}
$$
is the probability that a random walk starting at $v$ is at $w$ after $n$ steps.
A vector $\eta \in \ell^2_\sharp(V)$ defines
a local spectral measure for $M$, say $\nu_\eta$,
and a vertex $v \in V$ defines a vertex spectral measure for $M$, say $\nu_v = \nu_{\varepsilon_v}$.
A vertex $v$ is \textbf{$M$-dominant} if $\nu_\eta \prec \hskip-.1cm \prec\nu_v$
for all $\eta \in \ell^2_\sharp(V)$.
\par

When $G$ is a regular graph of some degree $k$,
the space $\ell^2_\sharp(V)$ and $\ell^2(V)$ are the same as sets,
scalar products are contant multiples of one another,
$\langle \eta_1 \sca \eta_2 \rangle_\sharp = \frac{1}{k} \langle \eta_1 \sca \eta_2 \rangle$,
and $M = \frac{1}{k} A$.
In this case, a vertex of $G$ is $M$-dominant if and only if it is $A$-dominant.
\par

However, this does not hold in general.
Consider for example an integer $n \ge 2$ and
the finite path $P_n$ with vertices $\{v_0, \hdots, v_{n-1} \}$.
Then $M$ has a simple eigenvalue $\cos \frac{ k \pi }{ n-1 }$
with eigenvector
$\eta_k = \big( \cos \frac{ kj \pi }{ n-1 } \big)_{j = 0, \hdots, n-1}$
for each $k \in \{0, \hdots, n-1\}$.
It follows that, 
when $n$ is even, all vertices of $P_n$ are $M$-dominant;
recall from Example~\ref{finitepaths} that,
when $n+1$ is not prime, some vertices are not $A$-dominant.
In contrast, when $n=9$, the only $M$-dominant vertices of $P_9$
are $v_0$ and $v_8$, but the $A$-dominant vertices are $v_0, v_2, v_6, v_8$.
\par

For the connected finite graphs we know without $A$-dominant vertices,
a close look shows that these graphs are also without $M$-dominant vertices.
It is open for us whether this holds for all graphs without $A$-dominant vertices.
Note that Propositions 3.4 and 3.7
hold for $M$ without changes.


\begin{thebibliography}{AkGW--08}


\bibitem[AkGM--13]{AkGM--13}
S.\ Akbari, E.\ Ghorbani, and A.\ Mahmoodi,
\emph{Nowhere-zero eigenvectors of graphs}.
Linear Multilinear Algebra \textbf{61} (2013), no.\ 2, 273--279.

\bibitem[Bigg--93]{Bigg--93}
N.\ Biggs,
\emph{Algebraic graph theory}, Second Edition.
Cambridge University Press, 1993.
[First edition 1974.]

\bibitem[BiLS--07]{BiLS--07}
T.\ Biyikoglu, J.\ Leydold, and P.F.\ Stadler,
\emph{Laplacian eigenvectors on graphs}.
Lecture Notes.in Math. \textbf{1915}, Springer 2007.


\bibitem[BoSm--20]{BoSm--20}
V.I.\ Bogachev and O.G.\ Smolyanov,
\emph{Real and functional analysis}.
Moscow Lectures \textbf{4}, Springer, 2020.

\bibitem[BrHa--12]{BrHa--12}
A.E.\ Brouwer and W.H.\ Haemers,
\emph{Spectra of graphs}.
Universitext, Springer, 2012.

\bibitem[CaCe--17]{CaCe--17}
A.\ Cafure and E.\ Cesaratto,
\emph{Irreducibility criteria for reciprocal polynomials and applications}.
Amer.\ Math.\ Monthly \textbf{124} (2017), no.\ 1, 37--53.

\bibitem[CoSi--57]{CoSi--57}
L.\ Collatz and U.\ Sinogowitz,
\emph{Spektren endlicher Grafen}.
Abh.\ Math.\ Sem.\ Univ.\ Hamburg \textbf{21} (1957), 63--77.

\bibitem[Conw--07]{Conw--07}
J.B.\ Conway,
\emph{A course in functional analysis}. Second Edition.
Graduate Texts in Mathematics \textbf{96}, Springer, 2007.

\bibitem[CvDS--80]{CvDS--80}
D.M.\ Cvetkovi\'c, M.\ Doob, and H.\ Sachs,
\emph{Spectra of graphs. Theory and applications.}
Academic Press, 1980.

\bibitem[CvRS--97]{CvRS--97}
D.M.\ Cvetkovi\'c, P.\ Rowlinson, and S.\ Simi\'c,
\emph{Eigenspaces of graphs.}
Cambridge University Press, 1997.

\bibitem[vDam--95]{vDam--95}
E.R.\ van Dam,
\emph{Regular graphs with four eigenvalues}.
Linear Algebra Appl.\ \textbf{226/228} (1995), 139--162.

\bibitem[DeRo--20]{DeRo--20}
D.\ DeFord and D.N.\ Rockmore,
\emph{On the spectrum of finite, rooted homogeneous trees}.
Linear Algebra Appl.\ \textbf{598} (2020), 165--185.

\bibitem[Dixm--69]{Dixm--69}
J.\ Dixmier,
\emph{Les alg\`ebres d'op\'erateurs dans l'espace hilbertien
(alg\`ebres de von Neumann)}, deuxi\`eme \'edition, revue et augment\'ee.
Gauthier-Villars, 1969.
[First Edition 1957, English translation North-Holland 1981.]

\bibitem[Doug--72]{Doug--72}
R.G.\ Douglas,
\emph{Banach algebra techniques in operator theory}.
Academic Press, 1972
[Second Edition, Springer, 1998].

\bibitem[DuSc--63]{DuSc--63}
N.\ Dunford and J.T.\ Schwartz,
\emph{Linear operators. Part II: Spectral theory. Self adjoint operators in Hilbert space}.
Interscience, 1963.

\bibitem[FiPi--83]{FiPi--83}
A.\ Fig\`a-Talamanca and M.A.\ Picardello,
\emph{Harmonic analysis on free groups}.
Marcel Dekker, 1983.

\bibitem[Gant--77]{Gant--77}
F.R.\ Gantmacher,
\emph{The theory of matrices}, Volume one.
Chelsea 1977.

\bibitem[Gods--84]{Gods--84}
C.\ Godsil,
\emph{Spectra of trees}.
In ``Convexity and graph theory (Jerusalem, 1981),
Ann.\ Discrete Math.\ \textbf{30}, North--Holland, 1984, pages 151--189.

\bibitem[GoMc--80]{GoMc--80}
C.D.\ Godsil and B.D.\ McKay,
\emph{Feasibility conditions for the existence of walk-regular graphs}.
Linear Algebra Appl.\ \textbf{30} (1980), 51--61.

\bibitem[GoMo--88]{GoMo--88}
C.D.\ Godsil and B.\ Mohar,
\emph{Walk generating functions and spectral measures of infinite graphs}.
Linear Algebra Appl.\ \textbf{107} (1988), 191--206.

\bibitem[GoRo--01]{GoRo--01}
C.D.\ Godsil and C.D.\ Royle,
\emph{Algebraic graph theory}.
Graduate Texts in Mathematics \textbf{207}, Springer, 2001.

\bibitem[Goli--21]{Goli--21}
L.\ Golinskii,
\emph{Spectra of infinite graphs: two methods of computation}.
Oper.\ Matrices \textbf{15} (2021), no.\ 3, 985--1029.

\bibitem[Grig--11]{Grig--11}
R.\ Grigorchuk,
\emph{Some problems of the dynamics of group actions on rooted trees}.
Proc.\ Steklov Inst.\ Math.\ \textbf{273} (2011), no.\ 1, 64--175.

\bibitem[Halm--67]{Halm--67}
P.R.\ Halmos,
\emph{A Hilbert space problem book}.
Van Nostrand, 1967.
[Second Edition, Springer, 1982.]

\bibitem[Kest--59]{Kest--59}
H.\ Kesten,
\emph{Symmetric random walks on groups}.
Trans.\ Amer.\ Math.\ Soc.\ \textbf{92} (1959), no.\ 2, 336--354.


\bibitem[Kowa--19]{Kowa--19}
E.\ Kowalski,
\emph{An introduction to expander graphs}.
Cours sp\'ecialis\'es \textbf{26},
Soci\'et\'e math\'ematique de France, 2019.


\bibitem[LeNy--15]{LeNy--15}
V.O.\ Lebid' and L.O.\ Nyzhnyk,
\emph{Spectral analysis of some graphs with infinite rays}.
Ukrainian Math.\ J.\ \textbf{66} (2015), no.\ 9, 1333--1345.

\bibitem[LiSi--22]{LiSi--22}
F.\ Liu and J.\ Siemons,
\emph{Unlocking the walk matrix of a graph}.
J.\ Algebraic Combin.\ \textbf{55} (2022), no.\ 3, 663--690.

\bibitem[Moha--82]{Moha--82}
B.\ Mohar,
\emph{The spectrum of an infinite graph}.
Linear Algebra Appl.\ \textbf{48} (1982), 245--256.

\bibitem[MoWo--89]{MoWo--89}
B.\ Mohar and W.\ Woess,
\emph{A survey on spectra of infinite graphs}.
Bull.\ London Math.\ Soc.\ \textbf{21} (1989), no\ 3, 209--234.

\bibitem[ORWo--19]{ORWo--19}
S.\ O'Rourke and P.M.\ Wood,
\emph{Low-degree factors of random polynomials}.
J.\ Theoret.\ Probab.\ \textbf{32} (2019), no.\ 2, 1076--1104.

\bibitem[Quin--09]{Quin--09}
J.-F.\ Quint,
\emph{Harmonic analysis on the Pascal graph}.
J.\ Funct.\ anal.\ \textbf{256} (2009), no.\ 10, 3409--3460.

\bibitem[Simo--11]{Simo--11}
B.\ Simon,
\emph{Szeg\"o's theorem and its descendants}.
Princeton Univ.\ Press, 2011.

\bibitem[Sim4--15]{Sim4--15}
B.\ Simon,
\emph{Operator theory}.
A comprehensive course in analysis, Part 4.
Amer.\ Math.\ Soc., 2015.

\bibitem[Szwa--88]{Szwa--88}
R.\ Szwarc,
\emph{An analytic series of irreducible representations of the free group}.
Ann.\ Inst.\ Fourier \textbf{38} (1988), no.\ 1, 87--110.

\bibitem[Woes--00]{Woes--00}
W.\ Woess,
\emph{Random walks on infinite graphs and groups}.
Cambridge University Press, 2000.


\bibitem[Yafa--17]{Yafa--17}
D.R.\ Yafaev,
\emph{A point interaction for the discrete Schr\"odinger operator
and generalized Chebyshev polynomials}.
J.\ Math.\ Phys.\ \textbf{58} (2017), no.\ 6, 063511, 24 pp.

\bibitem[YLZH--21]{YLZH--21}
Q.\ Yu, F.\ Liu, H.\ Zhang, and Z.\ Heng,
\emph{Note on graphs with irreducible characteristic polynomials}.
Linear Algebra Appl.\ \textbf{629} (2021), 72--86.



\end{thebibliography}
\end{document}